\documentclass[a4paper]{amsart}

\usepackage[utf8]{inputenc}
\usepackage[T1]{fontenc}

\usepackage{amsthm, amssymb, amsmath, amsfonts}

\usepackage[ocgcolorlinks, colorlinks=true, pdfstartview=FitV, linkcolor=blue, citecolor=blue, urlcolor=blue, pagebackref=false]{hyperref}

\usepackage{microtype}

\newtheorem{thm}{Theorem}[section]
\newtheorem{prop}[thm]{Proposition}
\newtheorem{lem}[thm]{Lemma}

\theoremstyle{remark}
\newtheorem{rem}[thm]{Remark}

\renewcommand{\le}{\leqslant}
\def\les{\lesssim}
\renewcommand{\ge}{\geqslant}
\renewcommand{\subset}{\subseteq}
\newcommand{\mcl}{\mathcal}
\newcommand{\E}{\mathbb{E}}

\newcommand{\cE}{\mathcal{E}}
\newcommand{\N}{\mathbb{N}}

\newcommand{\Ll}{\left}
\newcommand{\Rr}{\right}

\newcommand{\R}{\mathbb{R}}

\newcommand{\Z}{\mathbb{Z}}
\renewcommand{\P}{\mathbb{P}}

\newcommand{\td}{\tilde}
\newcommand{\eps}{\varepsilon}
\def\d{{\mathrm{d}}}

\renewcommand{\hom}{\mathrm{hom}}

\numberwithin{equation}{section}

\title[Pointwise two-scale expansion]{Pointwise two-scale expansion for parabolic equations with random coefficients}
\author{Yu Gu, Jean-Christophe Mourrat}

\address[Yu Gu]{Department of Mathematics, Building 380, Stanford University, Stanford, CA, 94305, USA}

\address[Jean-Christophe Mourrat]{ENS Lyon, CNRS, 46 allée d'Italie, 69007 Lyon, France}

\begin{document}
\begin{abstract}

We investigate the first-order correction in the homogenization of linear parabolic equations with random coefficients. In dimension $3$ and higher and for coefficients having a finite range of dependence, we prove a pointwise version of the two-scale expansion. A similar expansion is derived for elliptic equations in divergence form. 
The result is surprising, since it was not expected to be true without further symmetry assumptions on the law of the coefficients.

\bigskip


\medskip


\end{abstract}
\maketitle
%
%
%
%
%
%
%
%
\section{Introduction}
\label{s:intro}
\setcounter{equation}{0}

\subsection{Main result}

We are interested in parabolic equations in divergence form when $d\geq 3$:
\begin{equation}
\label{mainEq}
\left\{
\begin{array}{ll}
\displaystyle{\partial_t u_\eps(t,x,\omega)=\frac12\nabla\cdot (\tilde{a}(\frac{x}{\eps},\omega)\nabla u_\eps(t,x,\omega))   } &\text{on } \R_+\times \R^d,\\
\\
\displaystyle{u_\eps(0,x,\omega)=f(x) } & \text{on } \R^d,
\end{array}
\right.
\end{equation}
where $\omega\in \Omega$ denotes a particular random realization sampled from a probability space $(\Omega,\mathcal{F},\P)$, the function $f$ is bounded and smooth, and $\tilde{a}: \R^d\times \Omega\to \R^{d\times d}$ is a random field of symmetric matrices satisfying the uniform ellipticity condition
\begin{equation}
C^{-1}|\xi|^2\leq \xi^T \tilde{a}(x,\omega)\xi\leq C|\xi|^2.
\end{equation}
Standard homogenization theory shows that under the assumptions of stationarity and ergodicity on the random field $\tilde{a}(x,\omega)$, there exists a deterministic matrix $\bar{A}$ such that $u_\eps$ converges to the solution $u_{\hom}$ of a ``homogenized'' equation:
\begin{equation}
\label{limitEq}
\left\{
\begin{array}{ll}
\displaystyle{\partial_t u_{\hom}(t,x)=\frac12\nabla\cdot (\bar{A}\nabla u_{\hom}(t,x))   } &\text{on } \R_+\times \R^d,\\
\\
\displaystyle{u_{\hom}(0,x)=f(x) } & \text{on } \R^d.
\end{array}
\right.
\end{equation}

The goal of this paper is to further analyze the difference between $u_\eps(t,x,\omega)$ and $u_{\hom}(t,x)$, in a \emph{pointwise} sense. We assume that the coefficients $\td{a}$ have a short range of dependence (more precisely, that they can be written as a local function of a homogeneous Poisson point process). For each fixed $(t,x)$, we show that 
\begin{equation}
\label{e:main-prev}
u_\eps(t,x,\omega)-u_{\hom}(t,x)
=  \eps \nabla u_{\hom}(t,x)\cdot \td{\phi}(x/\eps,\omega)+o(\eps),
\end{equation}
where $\td{\phi}$ is the (stationary) corrector, and where $o(\eps)/\eps\to 0$ in $L^1(\Omega)$.

\subsection{Context}

There is a large body of literature on stochastic homogenization, starting from the work of Kozlov \cite{kozlov1979averaging} and Papanicolaou-Varadhan \cite{papanicolaou1979boundary} on divergence form operators. Their results show that as the correlation length of the random coefficients goes to zero, the operator converges in a certain sense to the one with constant coefficients. The qualitative convergence essentially comes from an ergodic theorem. In order to provide convergence rates, a quantification of ergodicity is required. The first quantitative result was given by Yurinskii \cite{yurinskii1986averaging}, where an algebraic rate was obtained. Other suboptimal results were obtained in \cite{mourrat2011variance}. Caffarelli and Souganidis considered nonlinear equations, and also derived an error estimate \cite{caffarelli2010rates}. 

Optimal results have started appearing only very recently, beginning with the groundbreaking work of Gloria and Otto \cite{gloria2011optimal,gloria2012optimal} and Gloria, Neukamm and Otto \cite{gloria2013quantification,gloria2014optimal}. Further developments include \cite{marahrens2013annealed,gloria2014quantitative,armstrong2,armstrong3,gno-reg,armstrong4}.

We would like in particular to draw the reader's attention to the results in \cite{gloria2014optimal}. There, linear elliptic equations in divergence-form on the $d$-dimensional torus~$\mathbb{T}$ are considered (so that there is no boundary layer), and a two-scale expansion is proved, in the sense that
$$
\Ll\|u_\eps(x,\omega) -u_{\hom}(x) - \eps \nabla u_{\hom}(x) \cdot \td{\phi}(x/\eps,\omega) \Rr\|_{H^1(\mathbb{T})} = O(\eps),
$$
with obvious notation for $u_\eps$ and $u_{\hom}$, and where $O(\eps)/\eps$ is bounded in $L^2(\Omega)$ uniformly over $\eps$. (Striclty speaking, the equations studied there are discrete, and a minor modification in the definition of $u_\eps$ is necessary in order to suppress the discretization error.) This statement is probably best understood as the summary of two estimates: one on $u_\eps$, and one on its gradient:
$$
\Ll(\int_{\mathbb{T}} \Ll|u_\eps(x,\omega) -u_{\hom}(x)\Rr|^2 \, dx\Rr)^{1/2} = O(\eps),
$$
$$
\Ll(\int_{\mathbb{T}} \Ll|\nabla u_\eps(x,\omega) - \nabla u_{\hom}(x) - \nabla \td{\phi}(x/\eps,\omega) \nabla u_{\hom}(x)\Rr|^2 \, dx\Rr)^{1/2} = O(\eps).
$$
In particular, it does not follow from this result that 
\begin{equation}
\label{first-order}
u_\eps(x,\omega) - u_{\hom}(x) - \eps \nabla u_{\hom}(x) \cdot \td{\phi}(x/\eps,\omega) = o(\eps).
\end{equation}
In fact, one of us (JCM) started this project with the belief that the expansion \eqref{first-order} was \emph{wrong} in general; that in order for it to be true, an additional symmetry property of the coefficients had to be assumed, a good candidate being the invariance of the law of the coefficients under the transformation $z \mapsto -z$. Even the weaker fact that 
\begin{equation}
\label{e:weak}
\E\{u_\eps(x,\omega) \} - u_{\hom}(x) = o(\eps)
\end{equation}
seemed a priori unlikely to be true in general. For the most part, this belief was based on three observations:
\begin{enumerate}
\item Numerical evidence, in the discrete setting, indicates that $\eps^{-1}\big(\E\{u_\eps(x,\omega)\} - u_{\hom}(x)\big) $ does not converge to $0$ for ``generic'' periodic environments, see \cite[Section~4.4.2 and Figure~15]{cemracs};
\item A simple toy model was proposed in \cite[Remark~4.4]{cemracs} to ``explain''  that $\eps^{-1}\big(\E\{u_\eps(x,\omega) \} - u_{\hom}(x)\big)$ should be of order $1$ in general: when summing i.i.d.\ random variables, the rate of convergence in the central limit theorem is generically of order $\eps$ when $\eps^{-2}$ random variables are summed; but it is of order $\eps^2$ when the law of the random variables is invariant under the transformation $z \mapsto -z$;
\item In the regime of small ellipticity contrast, Conlon and Fahim showed that the $\E\{u_\eps(x,\omega)\} - u_{\hom}(x) = O(\eps^2)$ when the law of the coefficients is invariant under the transformation $z \mapsto -z$, but they only show that it is $O(\eps)$ in general; see \cite[Theorem~1.2, Proposition~A.1, Remark~8 and Lemma~A.2]{conlon}.
\end{enumerate}

Despite these strong indications to the contrary, our result \eqref{e:main-prev} on the parabolic equation implies the corresponding result for the elliptic equation. That is, the expansion \eqref{first-order} is actually \emph{true in general} (i.e.\ without it being necessary to assume that the law of the coefficients is invariant under a transformation such as $z \mapsto -z$).

Why are there so convincing arguments to the contrary? It seems to us that the core of the matter is that the foregoing observations (1-3) all concern \emph{discrete} equations (i.e.\ where the underlying space is $\Z^d$), while our proof of \eqref{e:main-prev} and \eqref{first-order} applies to \emph{continuous} equations. Interestingly, we do not know how to prove our result (or the weaker statement \eqref{e:weak}) in the discrete setting without making use of an assumption such as the invariance of the law of the coefficients under the transformation $z \mapsto -z$.

\medskip

Finally, we would like to point out that while it is fairly easy to pass from a result on the parabolic equation to one on the elliptic equation, the converse does not seem to be possible. In fact, we are not aware of any previous ``two-scale expansion'' result for parabolic equations.

\subsection{The probabilistic approach}

From a probabilistic point of view, homogenizing a differential operator with random coefficients corresponds to proving an invariance principle for a random motion in random environment. Kipnis and Varadhan have developed a general central limit theorem for additive functionals of reversible Markov processes \cite{kipnis1986central}. A large class of random motions in random environment can be analyzed by following their approach, using also the idea of the ``medium seen from the moving particle'' (see \cite{komorowski2012fluctuations} and the references therein). The proof is based on a martingale decomposition and an application of the martingale central limit theorem (CLT).

In order to make this argument quantitative, two ingredients are necessary. One is a quantitative version of the martingale CLT; the other is a quantitative estimate on the speed of convergence to equilibrium of the medium seen from the particle. This route was already pursued in \cite{mourrat2012berry,mourrat2012kantorovich,gu2014fluctuations}. The quantitative martingale CLT developped in \cite{mourrat2012kantorovich} for general martingales was further explored in \cite{gu2014fluctuations}. It was shown there that by focusing on continuous martingales, one can express the first-order correction in the CLT in simple terms involving the quadratic variation of the martingale. This will provide us with a suitable quantitative martingale CLT. In addition, we will also need to assert that the process of the environment seen from the particle converges to equilibrium sufficiently fast. This question was first investigated in \cite{mourrat2011variance}, and we will borrow from there the idea that it is sometimes sufficient to understand the convergence to equilibrium of the environment as seen by a standard Brownian motion (independent of the environment). Furthermore, we will rely crucially on moment bounds on the corrector and on the gradients of the Green function recently obtained in \cite{gloria2014improved, gloria2014quantitative}. All these tools will enable us to identify a deterministic first-order correction to the expansion in \eqref{e:main-prev}, which we will finally show to be zero.

\subsection{Other relevant work}

The probabilistic approach is particularly well-suited for obtaining pointwise information such as \eqref{e:main-prev}. While such pointwise results are relatively rare, the precise behavior of more global random quantities has received considerable attention. In particular, a central limit theorem for the averaged energy density was derived in \cite{rossignol, nolen2011normal,biskup2014central}. The large-scale correlations and then the scaling limit of the corrector are investigated in \cite{mourrat2014correlation, mourrat2014CLT}. A comparable study of the scaling limit of the fluctuations of $u_\eps$ was performed in \cite{gu-mourrat}. We stress however that this result only characterizes the \emph{fluctuations} of $u_\eps$, but not the \emph{bias} $\E[u_\eps] - u_\hom$. The desire to understand the typical size of the bias (cf.\ \eqref{e:weak}) is what initiated our study. 


For other types of equations, e.g. a deterministic operator perturbed by a highly oscillatory random potential, fluctuations around homogenized limits have been analyzed in different contexts \cite{figari1982mean,bal2008central,bal2010homogenization,gu2014fluctuations}, see a review \cite{bal2013review}. From a probabilistic perspective, it corresponds to a random motion independent of the random environment.


\subsection{Organization and notation}

The rest of the paper is organized as follows. We make assumptions on the random field $\tilde{a}(x,\omega)$ and state the main results in Section \ref{s:assRe}. Then we present a standard approach to diffusions in random environments in Section \ref{s:setup}. Some key estimates of the correctors and the Green functions are contained in Section \ref{s:pCo}. The proof of the main results are presented in Sections \ref{s:r}, \ref{s:m} and \ref{s:ell}. 

We write $a\les b$ when $a\leq Cb$ with a constant $C$ independent of $\eps,t,x$. The normal distribution with mean $\mu$ and variance $\sigma^2$ is denoted by $N(\mu,\sigma^2)$, and $q_t(x)$ is the density of $N(0,t)$. The Fourier transform is defined by $\hat{f}(\xi)=\int_{\R^d}f(x)e^{-i\xi\cdot x}dx$. We will have two independent probability spaces with the associated expectations denoted by $\E,\E_B$ respectively. The expectation in the product probability space is then denoted by $\E\E_B$.

\section{Assumptions and main results}
\label{s:assRe}

Let $\mcl{M}$ be an arbitrary metric space equipped with its Borel $\sigma$-algebra, and let $\mu$ be a $\sigma$-finite measure on $\mcl{M}$. We let $\omega$ be a Poisson point process on $\mcl{M} \times \R^d$ with intensity measure $\d \mu(m) \, \d x$. We think of $\omega$ as an element of the probability space $(\Omega,\mcl{F},\P)$, where $\Omega$ is the collection of countable subsets of $\mcl{M} \times \R^d$, and $\mcl{F}$ is the smallest $\sigma$-algebra that makes the maps
$$
\left\{
\begin{array}{lll}
\Omega & \to & \N \cup \{+\infty\} \\
\omega & \mapsto & \mathrm{Card} (\omega \cap A) 
\end{array}
\right.
$$
measurable, for every measurable $A \subset \mcl{M} \times \R^d$. For a construction of such Poisson point processes, we refer to \cite[Section~2.5]{kingman}. For any measurable $S\subset \R^d$, we denote the $\sigma$-algebra generated by the Poisson point process restricted to $\mathcal{M}\times S$ by $\mathcal{F}_S$.

The group of translations of $\R^d$ can be naturally lifted to the space $\Omega$ by defining, for every $x \in \R^d$,
$$
\tau_x \omega = \{(m,x+z) \ : \ (m,z) \in \omega \} = \omega + (0,x).
$$
%
%
It is a classical result that $\{\tau_x,x\in\R^d\}$ satisfies the following properties:
\begin{enumerate}
\item \emph{Measure-preserving}: $\P \circ \tau_x=\P$.
\item \emph{Ergodicity}: if a measurable set $A \subset \Omega$ is such that for every $x \in \R^d$, $A = \tau_x(A)$, then $\P(A)\in\{0,1\}$.
\item \emph{Stochastic continuity}: for any $\delta>0$ and $f$ bounded measurable, 
\begin{equation*}
\lim_{h\to 0}\P\{|f(\tau_h\omega)-f(\omega)|\geq \delta\}=0.
\end{equation*}
\end{enumerate}

We denote the inner product and norm on $L^2(\Omega)$ by $\langle.,.\rangle$ and $\|.\|$ respectively, and define the operator $T_x$ on $L^2(\Omega)$ as $T_x f:=f\circ \tau_{-x}$.
The family $\{T_x,x\in \R^d\}$ forms a $d$-parameter group of unitary operators on $L^2(\Omega)$. Stochastic continuity implies that the group is strongly continuous, and ergodicity asserts that a function $f$ is constant if and only if $T_xf=f$ for all $x\in \R^d$. 

Let $\{D_k,k=1,\ldots,d\}$ be the generators of the group $\{T_x,x\in \R^d\}$. They correspond to differentiations in $L^2(\Omega)$ in the canonical directions denoted by $\{e_k,k=1,\ldots,d\}$. 
The gradient is then denoted by $D:=(D_1,\ldots,D_d)$, and we define the Sobolev space $H^1(\Omega)$ as the completion of smooth functions under the norm $\|f\|_{H^1}^2:=\langle f,f\rangle+\sum_{k=1}^d\langle D_kf,D_kf\rangle$.

Any function $f$ on $\Omega$ can be extended to a stationary random field $\tilde{f}(x,\omega):=f(\tau_{-x}\omega)$. The random coefficients $\tilde{a}(x,\omega)$ appearing in \eqref{mainEq} are given by $\tilde{a}(x,\omega)=a(\tau_{-x}\omega)$ for some measurable $a:\Omega\to \R^{d\times d}$. We further make the following assumptions on $a$:

\begin{enumerate}

\item \emph{Uniform ellipticity and smoothness}. For every $\omega \in \Omega$, $a(\omega)$ is a symmetric matrix satisfying 
\begin{equation}
\label{uniEll}
C^{-1}|\xi|^2\leq \xi^T a(\omega)\xi\leq C|\xi|^2
\end{equation} for some constant $C>0$.  Each entry $\tilde{a}_{ij}(x,\omega)=a_{ij}(\tau_{-x}\omega)$ has $C^2$ sample paths whose first and second order derivatives are uniformly bounded in $(x,\omega)$.

\item \emph{Local dependence}. 
There exists $C > 0$ such that for all $x\in \R^d$, $\tilde{a}(x,\omega)=a(\tau_{-x}\omega)$ is $\mathcal{F}_{\{y:|y-x|\leq C\}}$-measurable for some constant $C>0$.

\end{enumerate}

The coefficient field $a(\omega)$ can for instance be constructed by choosing a ``shape function'' $g: \mcl{M} \times \R^d \to E$ for some measurable vector space $E$ (e.g.\ the space of symmetric matrices) and a ``cut-off function'' $F : E \to \R^{d\times d}$ (that can be used to ensure uniform ellipticity), and letting
$$
a(\omega) = F\Ll(\sum_{(m,z) \in \omega} g(m,z)\Rr).
$$
The condition of local dependence on $a$ is guaranteed if $g(m,z)$ is non-zero only for $z$ varying in a compact set. As we will see below, the Poisson structure is only used to establish the covariance estimate \eqref{covEs1} and then prove Propositions~\ref{p:corCo} and \ref{p:corQu}. Although the law of the Poisson point process is invariant under transformations such as $z \mapsto -z$, this is of course not the case in general for the coefficient field $\td{a}(x,\omega)$ itself.



The following is our main theorem.
\begin{thm}
\label{t:main}
Assume $f\in \mathcal{C}_c^\infty(\R^d)$. For every $(t,x)$, there exists $C_\eps\to 0$ in $L^1(\Omega)$ such that 
\begin{equation}
\label{e:main}
u_\eps(t,x,\omega)-u_{\hom}(t,x)
=  \eps \nabla u_{\hom}(t,x)\cdot \phi(\tau_{-x/\eps}\omega)+\eps C_\eps.
\end{equation}
Here $\phi=(\phi_{e_1},\ldots,\phi_{e_d})$, where $\phi_{e_k}$ is the (zero-mean) stationary corrector in the canonical direction $e_k$.
\end{thm}

\begin{rem}
The existence of $\phi$ is given by Theorem~\ref{t:bdCo}. An examination of the proof reveals that the smoothness condition on $f$ can be relaxed. It suffices to assume that sufficiently many weak derivatives of $f$ belong to $L^2(\R^d)$ (i.e., the Fourier transform $\hat{f}$ is such that $\hat{f}(\xi)(1+|\xi|)^n$ is integrable for some large $n$).
\end{rem}

\begin{rem}
\label{}
It would be interesting to quantify the convergence of $\E[C_\eps]$ to $0$. We discuss a possible approach to show that $\E[C_\eps] \lesssim \sqrt{\eps}$ (up to logarithmic corrections) in Remark~\ref{r:exponent-improvement} below. 
\end{rem}

\begin{rem}
\label{}
To the best of our knowledge, Theorem~\ref{t:main} was not known even in the periodic case. We explain how to adapt our methods to this setting in Section~\ref{s:per}.
\end{rem}

Theorem~\ref{t:main} gives, for every $(t,x)$, the existence of some $C_\eps = C_\eps(t,x)$ such that \eqref{e:main} holds. Our proof actually shows more. In particular, for every $T >0$,  $\sup_{x \in \R^d, t \le T} \E\{|C_\eps(t,x)|\}$ tends to $0$ as $\eps$ tends to $0$, and we also obtain some control on the growth of this quantity as $T$ grows. Therefore, we can derive a similar result for elliptic equations, which we now describe more precisely.

Let $U_\eps(x,\omega)$ and $U_{\hom}(x)$ solve the following equations on $\R^d$ respectively 
\begin{eqnarray}
U_\eps(x,\omega)-\frac12\nabla\cdot(\tilde{a}(\frac{x}{\eps},\omega)\nabla U_\eps(x,\omega))&=&f(x),\\
U_{\hom}(x)-\frac12\nabla \cdot (\bar{A} \nabla U_{\hom}(x))&=&f(x).
\end{eqnarray}
\begin{thm}
\label{t:mainE}
Under the same assumption as in Theorem~\ref{t:main} and for every $x$, there exists $\tilde{C}_\eps\to 0$ in $L^1(\Omega)$ such that
\begin{equation*}
U_\eps(x,\omega)-U_{\hom}(x)
=\eps \nabla U_{\hom}(x)\cdot \phi(\tau_{-x/\eps}\omega)+\eps \tilde{C}_\eps.
\end{equation*}
\end{thm}

\section{Diffusions in random environments}
\label{s:setup}
\setcounter{equation}{0}

In this section, we present a standard approach to diffusions in random environments, including the process of the medium seen from the particle, corrector equations and the martingale decomposition. A complete introduction can be found in \cite[Chapter 9]{komorowski2012fluctuations}, so we do not present the details. 


For every fixed $\omega\in \Omega, x\in\R^d$ and $\eps>0$, we define the diffusion process $X_t^\omega$ on $\R^d$, starting from $x/\eps$,  by the It\^o stochastic differential equation
\begin{equation}
\label{disde}
dX_t^\omega=\tilde{b}(X_t^\omega,\omega)dt+\tilde{\sigma}(X_t^\omega,\omega)dB_t.
\end{equation}
Here, the drift $\tilde{b}=(\tilde{b}_1,\ldots,\tilde{b}_d)$ is defined by $\tilde{b}_i=\frac12\sum_{j=1}^d \partial_{x_j}\tilde{a}_{ji}$, the diffusion matrix is $\tilde{\sigma}=\sqrt{\tilde{a}}$, and the driving force $B_t=(B_t^1,\ldots,B_t^d)$ is a standard $d$-dimensional Brownian motion built on a different probability space $(\Sigma,\mathcal{A},\P_B)$ with the associated expectation $\E_B$. (Although we keep it implicit in the notation, note that the starting point of the diffusion depends on $\eps$.)

The medium or environment seen from the particle is the process taking values in $\Omega$ defined by
\begin{equation}
\omega_s:=\tau_{-X_s^\omega}\omega.
\end{equation}
The following lemma is taken from \cite[Proposition~9.8]{komorowski2012fluctuations}.
\begin{lem}
\label{lem:proEn}
$(\omega_s)_{s \ge 0}$ is a Markov process that is reversible and ergodic with respect to the measure $\P$. Its generator is given by $$L:=\frac12\sum_{i,j=1}^d D_i(a_{ij}D_j).$$
\end{lem}

The diffusively rescaled process $\eps X_{t/\eps^2}^\omega$ starts from $x$, with an infinitesimal generator given by
\begin{equation}
\label{infiGe}
L_\eps^\omega:=\frac12\sum_{i,j=1}^d\tilde{a}_{ij}(\frac{x}{\eps},\omega)\partial_{x_i}\partial_{x_j}+\frac{1}{\eps}\tilde{b}(\frac{x}{\eps},\omega)\cdot \nabla=\frac12\nabla\cdot (\tilde{a}(\frac{x}{\eps},\omega)\nabla).
\end{equation}
Hence, we can express the solution to \eqref{mainEq} as an average with respect to the diffusion process $\eps X_{t/\eps^2}^\omega$, i.e., for every fixed $\omega\in \Omega, t>0, x\in \R^d, \eps>0$, we have 
\begin{equation}
\label{proRe}
u_\eps(t,x,\omega)=\E_B\{f(\eps X_{t/\eps^2}^\omega)\}.
\end{equation}
With the above probabilistic representation, the problem reduces to an analysis of the asymptotic behavior of $\eps X_{t/\eps^2}^\omega$. In view of \eqref{disde}, the process can be written as
\begin{eqnarray*}
\eps X_{t/\eps^2}^\omega &=& x+\eps\int_0^{t/\eps^2}\tilde{b}(X_s^\omega,\omega) \, ds+\eps\int_0^{t/\eps^2}\tilde{\sigma}(X_s^\omega,\omega) \, dB_s\\
&=& x+ \eps \int_0^{t/\eps^2}b(\omega_s) \, ds+\eps\int_0^{t/\eps^2}\sigma(\omega_s) \, dB_s.
\end{eqnarray*}
It is clear that $b=(b_1,\ldots,b_d)$ with $b_i=\frac12\sum_{j=1}^d D_ja_{ji}$ and $\sigma=\sqrt{a}$.

The idea is to decompose the drift term $\eps \int_0^{t/\eps^2}b(\omega_s)ds$ as a martingale plus some small remainder. Since it is an additive functional of a stationary and ergodic Markov process, we can use the Kipnis-Varadhan method. For any $\lambda>0$, the $\lambda$-corrector in the direction of $\xi\in \R^d$, denoted by $\phi_{\lambda,\xi}$, is defined as the solution in $L^2(\Omega)$ to the following equation:
\begin{equation}
\label{corEqP}
(\lambda-L)\phi_{\lambda,\xi}=\xi\cdot b.
\end{equation}
By It\^o's formula, 
\begin{equation}
\begin{aligned}
\td{\phi}_{\lambda,\xi}(X_{t/\eps^2}^\omega,\omega) - \td{\phi}_{\lambda,\xi}(X_{0}^\omega,\omega) = &\int_0^{t/\eps^2} L^\omega_1 \td{\phi}_{\lambda,\xi}(X_s^\omega,\omega) \, ds \\
&+ \sum_{i,j = 1}^d \int_0^t \partial_{x_i} \td{\phi}_{\lambda,\xi}(X_s^\omega,\omega) \td{\sigma}_{ij}(X_s^\omega,\omega) \, dB_s^j.
\end{aligned}
\end{equation}
Hence, the projection on $\xi$ of the drift term can be decomposed as
\begin{equation*}
\begin{aligned}
\int_0^{t/\eps^2}(\xi\cdot b)(\omega_s) \, ds=&\int_0^{t/\eps^2}\lambda\phi_{\lambda,\xi}(\omega_s) \, ds-\phi_{\lambda,\xi}(\omega_{t/\eps^2})+\phi_{\lambda,\xi}(\omega_0)\\
&+\sum_{i,j=1}^d \int_0^{t/\eps^2}D_i\phi_{\lambda,\xi}(\omega_s)\sigma_{ij}(\omega_s) \, dB_s^j,
\end{aligned}
\end{equation*}
so the projection on $\xi$ of the rescaled process admits the following representation:
\begin{equation}
\label{e:decomp}
 \xi\cdot \left( \eps X_{t/\eps^2}^\omega \right)=\xi\cdot x+R_t^\eps(\lambda)+M_t^\eps(\lambda),
\end{equation} where the remainder $R_t^\eps(\lambda)$ and the martingale $M_t^\eps(\lambda)$ are given by
\begin{eqnarray}
R_t^\eps(\lambda)&:=&\eps\int_0^{t/\eps^2}\lambda\phi_{\lambda,\xi}(\omega_s) \, ds-\eps\phi_{\lambda,\xi}(\omega_{t/\eps^2})+\eps\phi_{\lambda,\xi}(\omega_0),\\
M_t^\eps(\lambda)&:=&\sum_{j=1}^d \eps\int_0^{t/\eps^2}\sum_{i=1}^d(D_i\phi_{\lambda,\xi}(\omega_s)+\xi_i)\sigma_{ij}(\omega_s) \, dB_s^j.
\end{eqnarray}

We point out that equation \eqref{corEqP} on the probability space $\Omega$ corresponds to the following PDE on the physical space $\R^d$:
\begin{equation}
\label{corEqR}
(\lambda-L_1^\omega)\tilde{\phi}_{\lambda,\xi}=\xi\cdot \tilde{b},
\end{equation}
where we recall that $\tilde{\phi}_{\lambda,\xi}(x,\omega)=\phi_{\lambda,\xi}(\tau_{-x}\omega)$. Letting $G^\omega_\lambda(x,y)$ be the Green function associated with $\lambda-L_1^\omega$, we have the integral representation
\begin{equation}
\phi_{\lambda,\xi}(\tau_{-x}\omega)=\int_{\R^d}G^\omega_\lambda(x,y)\xi\cdot b(\tau_{-y}\omega) \, dy.
\end{equation}

We briefly discuss the proof of homogenization, see \cite[Chapter 9]{komorowski2012fluctuations} for details. For the remainder, it can be shown that $\lambda\langle\phi_{\lambda,\xi},\phi_{\lambda,\xi}\rangle\to 0$ as $\lambda\to 0$, so by applying Lemma~\ref{lem:proEn} and choosing $\lambda=\eps^2$, we obtain $\E\E_B\{|R_t^\eps(\lambda)|^2\}\to 0$ as $\eps\to 0$. For the martingale, we can first show that $D\phi_{\lambda,\xi}$ converges in $L^2(\Omega)$, with the limit formally written as $D\phi_\xi$. Then by a martingale central limit theorem, $M_t^\eps(\lambda)$ converges in distribution to a Gaussian with mean zero and variance $\sigma_\xi^2:=\xi^T\bar{A}\xi$, where the homogenized matrix $\bar{A}$ is given by 
\begin{equation}
\label{hoMa}
\bar{A}_{ij}=\E\{(e_i+D \phi_{e_i})^Ta( e_j+D\phi_{e_j})\}.
\end{equation}
We can express the solution \eqref{proRe} in the Fourier domain using~\eqref{e:decomp} as
\begin{equation}
u_\eps(t,x,\omega)=\frac{1}{(2\pi)^d}\int_{\R^d} \hat{f}(\xi)e^{i\xi \cdot x}\E_B\{e^{i R_t^\eps(\lambda)} e^{iM_t^\eps(\lambda)}\} \, d\xi.
\end{equation}
By the convergence of $R_t^\eps(\lambda)\to 0$ and $M_t^\eps(\lambda)\to  N(0,\sigma_\xi^2)$, it can be shown that
\begin{equation}
u_\eps(t,x,\omega)\to u_{\hom}(t,x)=\frac{1}{(2\pi)^d}\int_{\R^d}\hat{f}(\xi)e^{i\xi \cdot x}e^{-\frac12\xi^T\bar{A}\xi t} \, d\xi
\end{equation}
in probability.

\section{Properties of correctors and functionals of the environment seen from the particle}
\label{s:pCo}

In this section, we first present some key estimates on the corrector $\phi_{\lambda,\xi}$ and the Green function $G_\lambda^\omega(x,y)$. 
Then we analyze the decorrelation rate of certain functionals of the corrector by an application of the spectral gap inequality. In the end, we estimate the variance decay of functionals of the environmental process by a comparison of resolvents. Throughout the section, $\xi$ is a fixed vector in $\R^d$.

The following two theorems are borrowed from \cite[Proposition~1]{gloria2014quantitative} and \cite[Corollary~1.5]{gloria2014improved}.
\begin{thm}[\cite{gloria2014quantitative}]
\label{t:bdCo}
Recall that $d\geq 3$. There exists $\phi_{\xi}\in H^1(\Omega)$ such that $\phi_{\lambda,\xi}\to \phi_\xi$ in $H^1(\Omega)$ as $\lambda$ tends to $0$. Furthermore, the $p$-th moments of $\phi_{\lambda,\xi},D\phi_{\lambda,\xi},\phi_\xi,D\phi_\xi$ are uniformly bounded in $\lambda$ for any $p<\infty$.
\end{thm}

\begin{rem}
From \eqref{corEqP}, it is clear that $\E\{\phi_{\lambda,\xi}\}=0$, so $\E\{\phi_\xi\}=0$.
\end{rem}

\begin{rem}
For the gradient of the corrector $D\phi_{\lambda,\xi}$, \cite[Proposition~1]{gloria2014quantitative} proves 
$$
\E\Ll\{\Ll(\int_{|x|\leq 1}|\nabla\tilde{\phi}_{\lambda,\xi}(x,\omega)|^2dx\Rr)^p\Rr\}\leq C_p
$$ for any $p>0$, i.e., a high moment bound of some spatial average. This can be improved with additional regularity assumptions on $\tilde{a}$. Recall that for almost every~$\omega$, $\tilde{\phi}_{\lambda,\xi}(x,\omega)$ is the weak solution to
\begin{equation*}
\lambda \tilde{\phi}_{\lambda,\xi}(x,\omega)-\frac12\nabla\cdot \tilde{a}(x,\omega)(\xi + \nabla\tilde{\phi}_{\lambda,\xi}(x,\omega))=0,
\end{equation*}
and since the sample path of $\tilde{a}_{ij}(x,\omega)$ is $C^2$ and hence H\"older continuous (uniformly over $\omega$), the following estimate is given by standard H\"older regularity theory \cite[Theorems~3.13 and 3.1]{han1997elliptic}
\begin{equation*}
|\nabla\tilde{\phi}_{\lambda,\xi}(0,\omega)|^2\leq C\left(1+\int_{|x|\leq 1} |\tilde{\phi}_{\lambda,\xi}(x,\omega)|^2dx+\int_{|x|\leq 1} |\nabla\tilde{\phi}_{\lambda,\xi}(x,\omega)|^2dx\right),
\end{equation*}
with the constant $C$ independent of $\omega$ and $\lambda \le 1$. By taking expectation, we derive a bound on the $L^p$ norm of $D\phi_{\lambda,\xi}$ that is uniform in $\lambda \le 1$.
\end{rem}

\begin{thm}[\cite{gloria2014improved,armstrong2}]
\label{t:bdGr}
Recall that $d\geq 3$. For every $p>0$, there exists $C_p < \infty$ such that for every $\lambda \ge 0$ and $x,y\in\R^d$,
\begin{eqnarray*}
\E\{|\nabla_x G^\omega_\lambda(x,y)|^p\}^{\frac{1}{p}}&\leq& \frac{C_p}{|x-y|^{d-1}},\\
\E\{|\nabla_x\nabla_y G^\omega_\lambda(x,y)|^p\}^{\frac{1}{p}}&\leq& \frac{C_p}{|x-y|^{d}},
\end{eqnarray*}
where the constant $C_p>0$ does not depend on $\lambda$, and $\nabla_x\nabla_y$ denotes the mixed second order derivatives.
\end{thm}

The Poisson structure that we assume enables us to decompose the randomness into i.i.d.\ random variables, i.e., we have $\omega=\{\eta_k,k\in \Z^d\}$ with $\eta_k$ the Poisson point process restricted on $\mathcal{M}\times \{k+[0,1)^d\}$. In this way, we can use a spectral gap inequality given by \cite[Lemma~1]{gloria2013quantification} to estimate the decorrelation rates of functions on $\Omega$. For any $f\in L^2(\Omega)$ with $\E\{f\}=0$, the inequality shows 
\begin{equation}
\E\{f^2\}\leq \sum_{k\in \Z^d}\E\{|\partial_k f|^2\},
\end{equation}
with $\partial_k f:=f-\E\{f|\{\eta_i, i\neq k\}\}$ describing the dependence of $f$ on $\eta_k$.

By following the same argument, a covariance estimate can be derived, i.e., for any $f,g\in L^2(\Omega)$ with $\E\{f\}=\E\{g\}=0$, we have
\begin{equation}
\label{covEs1}
|\E\{fg\}|\leq \sum_{k\in \Z^d}\sqrt{\E\{|\partial_k f|^2\}}\sqrt{\E\{|\partial_k g|^2\}}.
\end{equation}

We further claim that 
\begin{equation}
\label{deEs}
\E\{|\partial_k f|^2\}= \frac{1}{2} \,  \E\{|f-f_k|^2\}.
\end{equation}
Here $f_k(\omega):=f(\omega_k)$ with $\omega_k:=\{\eta_i,i\neq k\}\cup \{\tilde{\eta}_k\}$ and $\tilde{\eta}_k$ an independent copy of $\eta_k$, i.e., $\omega_k$ is a perturbation of $\omega$ at $k$. First, since conditional expectation is an $L^2$ projection, we have $\E\{|\partial_k f|^2\}=\E\{f^2\}-\E\{|\E\{f|\{\eta_i, i\neq k\}\}|^2\}$. Secondly, $\E\{|f-f_k|^2\}=2\E\{f^2\}-2\E\{ff_k\}$ and $\E\{ff_k\}=\E\{|\E\{f|\{\eta_i, i\neq k\}\}|^2\}$ by conditioning on $\{\eta_i,i\neq k\}$. So \eqref{deEs} is proved.

Combining \eqref{covEs1} and \eqref{deEs}, we obtain
\begin{equation}
\label{covEs2}
|\E\{fg\}|\leq \sum_{k\in \Z^d}\sqrt{\E\{|f-f_k|^2\}}\sqrt{\E\{|g-g_k|^2\}}.
\end{equation}
This will be our main tool to estimate the decorrelation rate of functionals on $\Omega$.

\begin{rem}

\label{r:checker}
The covariance estimate also holds for the random checkerboard structure, e.g., let $\tilde{a}(x,\omega)=\eta_k$ if $x-k\in[0,1)^d$, with $\{\eta_k, k\in \Z^d\}$ i.i.d.\ matrix-valued random variables. However, in that case $\tilde{a}(x,\omega)$ is only stationary with respect to shifts in $\Z^d$,  and such situations are not covered by Theorems~\ref{t:bdCo} and \ref{t:bdGr}.
\end{rem}

The following is an estimate of the decorrelation rate of $\phi_\xi$.

\begin{prop}
\label{p:corCo}
$|\E\{\phi_\xi(\tau_0\omega)\phi_\xi(\tau_{-x}\omega)\}|\les |\xi|^2(1\wedge \frac{1}{|x|^{d-2}})$.
\end{prop}

\begin{proof}
By Theorem~\ref{t:bdCo}, $\phi_{\lambda,\xi}\to \phi_\xi$ in $L^2(\Omega)$, so we only need to show that the estimate holds for $\phi_{\lambda,\xi}$ with an implicit constant independent of $\lambda$. Clearly, it suffices to consider $|x|$ sufficiently large.

By \eqref{covEs2} we have
\begin{equation}
\begin{aligned}
&|\E\{\phi_{\lambda,\xi}(\tau_0\omega)\phi_{\lambda,\xi}(\tau_{-x}\omega)\}|\\
\leq &\sum_{k\in \Z^d}\sqrt{\E\{|\phi_{\lambda,\xi}(\tau_0\omega)-\phi_{\lambda,\xi}(\tau_0\omega_k)|^2\}}\sqrt{\E\{|\phi_{\lambda,\xi}(\tau_{-x}\omega)-\phi_{\lambda,\xi}(\tau_{-x}\omega_k)|^2\}},
\end{aligned}
\end{equation}
where $\omega_k$ is obtained by replacing $\eta_k$ in $\omega$ by an independent copy $\tilde{\eta}_k$.

Now we only need to control $\E\{|\phi_{\lambda,\xi}(\tau_{-x}\omega)-\phi_{\lambda,\xi}(\tau_{-x}\omega_k)|^2\}$ for $x\in \R^d, k\in \Z^d$. Since it is bounded, we consider the case when $|x-k|$ is large. Recall that we write $\tilde{\phi}_{\lambda,\xi}(x,\omega)=\phi_{\lambda,\xi}(\tau_{-x}\omega)$, and that
\begin{eqnarray}
\lambda\tilde{\phi}_{\lambda,\xi}(x,\omega)-\frac12\nabla\cdot (\tilde{a}(x,\omega)\nabla\tilde{\phi}_{\lambda,\xi}(x,\omega))&=&\xi\cdot \tilde{b}(x,\omega),\\
\lambda\tilde{\phi}_{\lambda,\xi}(x,\omega_k)-\frac12\nabla\cdot (\tilde{a}(x,\omega_k)\nabla\tilde{\phi}_{\lambda,\xi}(x,\omega_k))&=&\xi\cdot \tilde{b}(x,\omega_k).
\end{eqnarray}
As a consequence,
\begin{equation}
\label{e:diff-phi}
\begin{aligned}
&\tilde{\phi}_{\lambda,\xi}(x,\omega)-\tilde{\phi}_{\lambda,\xi}(x,\omega_k)\\
=&\int_{\R^d}G_\lambda^\omega(x,y)\left(\xi\cdot(\tilde{b}(y,\omega)-\tilde{b}(y,\omega_k))+\frac12\nabla\cdot (\tilde{a}(y,\omega)-\tilde{a}(y,\omega_k))\nabla \tilde{\phi}_{\lambda,\xi}(y,\omega_k)\right) \, dy\\
=&-\int_{\R^d}\nabla_y G_\lambda^\omega(x,y)\left(\frac12 (\tilde{a}(y,\omega)-\tilde{a}(y,\omega_k))\xi+\frac12(\tilde{a}(y,\omega)-\tilde{a}(y,\omega_k))\nabla \tilde{\phi}_{\lambda,\xi}(y,\omega_k)\right) \, dy,
\end{aligned}
\end{equation}
since $\xi \cdot \td{b} = \frac12 \nabla \cdot (\td{a} \xi)$. By the assumptions on $a$, $\tilde{a}(y,\omega)-\tilde{a}(y,\omega_k)=0$ when $|y-k|\geq C$ for some constant $C$, so
\begin{equation}
|\tilde{\phi}_{\lambda,\xi}(x,\omega)-\tilde{\phi}_{\lambda,\xi}(x,\omega_k)|\les \int_{|y-k|\leq C}|\nabla_y G_\lambda^\omega(x,y)|(|\xi|+|\nabla \tilde{\phi}_{\lambda,\xi}(y,\omega_k)|) \, dy,
\end{equation}
which implies
\begin{equation}
\label{pfcorCo}
\begin{aligned}
\E\{|\tilde{\phi}_{\lambda,\xi}(x,\omega)-\tilde{\phi}_{\lambda,\xi}(x,\omega_k)|^2\}\les & |\xi|^2\int_{|y-k|\leq C}\E\{|\nabla_y G_\lambda^\omega(x,y)|^2\} \, dy\\
& + \int_{|y-k|\leq C}\sqrt{\E\{|\nabla_y G_\lambda^\omega(x,y)|^4\}}\sqrt{\E\{|\nabla\tilde{\phi}_{\lambda,\xi}(y,\omega_k)|^4\}} \, dy.
\end{aligned}
\end{equation}
By Theorem~\ref{t:bdCo} and the fact that $\tilde{\phi}_{\lambda,\xi}$ is linear in $\xi$, we first observe that
\begin{equation}
\sqrt{\E\{|\nabla \tilde{\phi}_{\lambda,\xi}(y,\omega_k)|^4\}}\les |\xi|^2,
\end{equation}
then we apply Theorem~\ref{t:bdGr} on the r.h.s.\ of \eqref{pfcorCo} to derive 
\begin{equation}
\sqrt{\E\{|\phi_{\lambda,\xi}(\tau_{-x}\omega)-\phi_{\lambda,\xi}(\tau_{-x}\omega_k)|^2\}}\les  |\xi|( 1\wedge \frac{1}{|x-k|^{d-1}}).
\end{equation}

Now we have
\begin{equation}
|\E\{\phi_{\lambda,\xi}(\tau_0\omega)\phi_{\lambda,\xi}(\tau_{-x}\omega)\}|\les |\xi|^2\sum_{k\in \Z^d} (1\wedge \frac{1}{|k|^{d-1}})(1\wedge \frac{1}{|x-k|^{d-1}})\les \frac{|\xi|^2}{|x|^{d-2}},
\end{equation}
where the last inequality comes from Lemma~\ref{l:conL}. The proof is complete.
\end{proof}

Define 
\begin{equation}
\label{defpsi}
\begin{aligned}
\psi_\xi:=&(\xi+D\phi_\xi)^Ta(\xi+D\phi_\xi)-\xi^T\bar{A}\xi\\
=&\sum_{i,j=1}^d \xi_i\xi_j\left( (e_i+D\phi_{e_i})^Ta(e_j+D \phi_{e_j})-\bar{A}_{ij}\right),
\end{aligned}
\end{equation}
by the definition of the homogenized matrix $\bar{A}$ in \eqref{hoMa}, $\psi$ has mean zero and we can write it as $\psi_\xi=\sum_{i,j=1}^d\xi_i\xi_j\psi_{ij}$ with 
\begin{equation}
\psi_{ij}:=(e_i+D\phi_{e_i})^Ta(e_j+D \phi_{e_j})-\bar{A}_{ij}.
\end{equation}

The following is an estimate of the decorrelation rate of $\psi_\xi$.
\begin{prop}
\label{p:corQu}
$|\E\{\psi_\xi(\tau_0\omega)\psi_\xi(\tau_{-x}\omega)\}|\les |\xi|^4( 1\wedge \frac{\log(2+|x|)}{|x|^d})$.
\end{prop}

\begin{proof}
First we define $\psi_{\lambda,\xi}:=(\xi+D\phi_{\lambda,\xi})^Ta(\xi+D\phi_{\lambda,\xi})-\xi^T\bar{A}_\lambda\xi$, where $\bar{A}_\lambda$ is chosen so that $\psi_{\lambda,\xi}$ has zero mean. By Theorem~\ref{t:bdCo}, $\psi_{\lambda,\xi}\to \psi_\xi$ in $L^2(\Omega)$, so we only need to  consider $\psi_{\lambda,\xi}$ and show that the estimate holds uniformly in $\lambda$.  

Similarly, we apply \eqref{covEs2} to obtain
\begin{equation}
\begin{aligned}
&|\E\{\psi_{\lambda,\xi}(\tau_0\omega)\psi_{\lambda,\xi}(\tau_{-x}\omega)\}|\\
\leq &\sum_{k\in \Z^d}\sqrt{\E\{|\psi_{\lambda,\xi}(\tau_0\omega)-\psi_{\lambda,\xi}(\tau_0\omega_k)|^2\}}\sqrt{\E\{|\psi_{\lambda,\xi}(\tau_{-x}\omega)-\psi_{\lambda,\xi}(\tau_{-x}\omega_k)|^2\}},
\end{aligned}
\end{equation}
with $\omega_k$ the perturbation of $\omega$ at $k$.

For any vector $x_i,y_i\in \R^d$ and matrix $A_i\in \R^{d\times d}$, $i=1,2$, we have
\begin{equation}
|x_1^TA_1y_1-x_2^TA_2y_2|\leq |x_1-x_2|\cdot|y_1|\cdot \|A_1\|+|x_2|\cdot |y_1|\cdot \|A_1-A_2\|+|x_2|\cdot |y_1-y_2|\cdot \|A_2\|,
\end{equation}
with $\|.\|$ denoting the matrix norm here, so by the moment bounds of $D\phi_{\lambda,\xi}$, we derive
\begin{equation*}
\begin{aligned}
\E\{|\psi_{\lambda,\xi}(\tau_{-x}\omega)-\psi_{\lambda,\xi}(\tau_{-x}\omega_k)|^2\}\les |\xi|^4 &\sqrt{\E\{\|a(\tau_{-x}\omega)-a(\tau_{-x}\omega_k)\|^4\}}\\
&+|\xi|^2\sqrt{\E\{|D\phi_{\lambda,\xi}(\tau_{-x}\omega)-D\phi_{\lambda,\xi}(\tau_{-x}\omega_k)|^4\}}.
\end{aligned}
\end{equation*} 

First, $\sqrt{\E\{\|a(\tau_{-x}\omega)-a(\tau_{-x}\omega_k)\|^4\}}\les 1_{|x-k|\leq C}$ by the local dependence of $a$ on $\omega$. 

Secondly, recalling \eqref{e:diff-phi}, 
\begin{equation*}
\begin{aligned}
&\partial_{x_i}\tilde{\phi}_{\lambda,\xi}(x,\omega)-\partial_{x_i}\tilde{\phi}_{\lambda,\xi}(x,\omega_k)\\
=&-\int_{\R^d}\nabla_y \partial_{x_i}G_\lambda^\omega(x,y)\left(\frac12 (\tilde{a}(y,\omega)-\tilde{a}(y,\omega_k))\xi+\frac12(\tilde{a}(y,\omega)-\tilde{a}(y,\omega_k))\nabla \tilde{\phi}_{\lambda,\xi}(y,\omega_k)\right) \, dy.
\end{aligned}
\end{equation*}
By the same discussion as in the proof of Proposition~\ref{p:corCo}, we obtain 
\begin{equation}
\sqrt{\E\{|\nabla\tilde{\phi}_{\lambda,\xi}(x,\omega)-\nabla\tilde{\phi}_{\lambda,\xi}(x,\omega_k)|^4\}}\les |\xi|^2(1\wedge \frac{1}{|x-k|^{2d}}).
\end{equation}

To summarize, since $D\phi_{\lambda,\xi}(\tau_{-x}\omega)=\nabla\tilde{\phi}_{\lambda,\xi}(x,\omega)$, we have 
\begin{equation*}
\E\{|\psi_{\lambda,\xi}(\tau_{-x}\omega)-\psi_{\lambda,\xi}(\tau_{-x}\omega_k)|^2\}\les |\xi|^4( 1_{|x-k|\leq C}+1\wedge \frac{1}{|x-k|^{2d}}),
\end{equation*} so
\begin{equation*}
|\E\{\psi_{\lambda,\xi}(\tau_0\omega)\psi_{\lambda,\xi}(\tau_{-x}\omega)\}|\les |\xi|^4\sum_{k\in \Z^d}(1\wedge \frac{1}{|k|^d})(1\wedge \frac{1}{|x-k|^d})\les |\xi|^4\frac{\log(2+|x|)}{|x|^d},
\end{equation*}
where the last inequality comes from Lemma~\ref{l:conL}. The proof is complete.
\end{proof}

For any $f\in L^2(\Omega)$ with $\E\{f\}=0$, we are interested in the variance decay of 
\begin{equation}
f_t:=\E_B\{f(\omega_t)\}.
\end{equation}
Since $\omega_t=\tau_{-X_t^\omega}\omega$ and $X_t^\omega$ is driven by the generator $L_1^\omega=\frac12\nabla\cdot (\tilde{a}(x,\omega)\nabla)$ with $\tilde{a}$ being strictly positive definite, heuristically $X_t^\omega$ should spread at least as fast as a Brownian motion with a sufficiently small diffusion constant. In other words, letting $f_t^o:=\E_B\{f(\omega_t^o)\}$ with $\omega_t^o=\tau_{-B_s}\omega$, we expect the decay to $0$ of $f_t$ to be at least as fast as that of $f_t^o$ (up to rescaling the time by a suitable constant). The following result is a precise statement of this idea (see \cite[Lemma~5.1]{mourrat2011variance} for a classical proof).
\begin{prop}
\label{p:conRe}
For any $\lambda\ge0$, 
\begin{equation*}
\int_0^\infty e^{-\lambda t}\E\{|f_{t}|^2\} \, dt\leq C\int_0^\infty e^{-\lambda t}\E\{|f_{t}^o|^2\} \, dt.
\end{equation*}
The constant $C>0$ only depends on the ellipticity constant in \eqref{uniEll}.
\end{prop}

For $f=\phi_\xi$ or $\psi_\xi$, the following results holds.

\begin{prop}
\label{p:vdphi}
\begin{equation*}
\E\{|\E_B\{\phi_\xi(\omega_{t})\}|^2\}\les |\xi|^2
\Ll|
\begin{array}{ll}
t^{-\frac12} & \text{if } d =  3, \\
t^{-1}\log(2+t) & \text{if } d =  4, \\
t^{-1} & \text{if } d \ge 5.
\end{array}
\Rr.
\end{equation*}
\end{prop}

\begin{proof}
First, for any $f$ we have
\begin{equation*}
\E\{|f_{t/2}^o|^2\}=\E\{|\E_B\{f(\tau_{-B_{t/2}}\omega)\}|^2\}=\E\{\E_{B^1,B^2}\{f(\tau_{-B^1_{t/2}}\omega)f(\tau_{-B^2_{t/2}}\omega)\}\},
\end{equation*}
where $B^1,B^2$ are two independent Brownian motions and $\E_{B^1,B^2}$ denotes the average with respect to them. 

Next let $f=\phi_\xi$ and $R_{\phi_\xi}$ be the covariance function of $\phi_\xi$(and recalling that $q_t$ is the density of the law $N(0,t)$), we obtain 
\begin{equation*}
\begin{aligned}
\E\{|f_{t/2}^o|^2\}=&\E_{B^1,B^2}\{R_{\phi_\xi}(B^1_{t/2}-B^2_{t/2})\}=\int_{\R^d}R_{\phi_\xi}(x)q_t(x) \, dx\\
=&\int_{\R^d}R_{\phi_\xi}(\sqrt{t}x)q_1(x) \, dx
\les |\xi|^2\int_{\R^d} 1\wedge \frac{1}{|\sqrt{t}x|^{d-2}} q_1(x) \, dx\les |\xi|^2(1\wedge \frac{1}{t^{\frac{d}{2}-1}}),
\end{aligned}
\end{equation*}
where we used the result $|R_{\phi_\xi}(x)|\les |\xi|^2(1\wedge |x|^{2-d})$ given by Proposition~\ref{p:corCo}.

Since $\E\{|f_{t/2}|^2\}$ decreases in $t$, from Proposition~\ref{p:conRe} we have
\begin{equation}
\E\{|f_{t/2}|^2\}\leq \frac{C\lambda \int_0^\infty e^{-\lambda s}\E\{|f_{s/2}^o|^2\} \, ds}{1-e^{-\lambda t}}\les \frac{C\lambda |\xi|^2\int_0^\infty e^{-\lambda s}(1\wedge s^{-\frac{d}{2}+1}) \, ds}{1-e^{-\lambda t}}
\end{equation}
for any $\lambda>0$. We can choose $\lambda=1/t$ on the r.h.s.\ of the above display and derive
\begin{equation}
\E\{|f_{t/2}|^2\}\les|\xi|^2
\Ll|
\begin{array}{ll}
t^{-\frac12} & \text{if } d =  3, \\
t^{-1}\log(2+t) & \text{if } d =  4, \\
t^{-1} & \text{if } d \ge 5.
\end{array}
\Rr.
\end{equation}
The proof is complete.
\end{proof}

\begin{prop}
\label{p:vdpsi}
$\int_0^\infty \E\{|\E_B\{\psi_\xi(\omega_t)\}|^2\} \, dt\les |\xi|^4$.
\end{prop}

\begin{proof}
Let $f=\psi_\xi$, by Proposition~\ref{p:conRe} we have
\begin{equation}
\int_0^\infty \E\{|f_{t}|^2\} \, dt\leq C\int_0^\infty \E\{|f_{t}^o|^2\} \, dt,
\end{equation}
so we only need to prove that $\int_0^\infty \E\{|f_{t}^o|^2\} \, dt\les |\xi|^4$. Let $R_{\psi_\xi}$ be the covariance function of $\psi_\xi$. By the same argument as in Proposition~\ref{p:vdphi},
\begin{equation}
\int_0^\infty \E\{|f_{t}^o|^2\} \, dt=\int_0^\infty \int_{\R^d}R_{\psi_\xi}(x)q_{2t}(x) \, dxdt.
\end{equation}
By Proposition~\ref{p:corQu}, $|R_{\psi_\xi}(x)|\les |\xi|^4(1\wedge |x|^{-d}\log(2+|x|))$, so after integrating in $t$ we obtain
\begin{equation}
\int_0^\infty \int_{\R^d}R_{\psi_\xi}(x)q_{2t}(x) \, dxdt\les \int_{\R^d}\frac{|R_{\psi_\xi}(x)|}{|x|^{d-2}} \, dx\les |\xi|^4
\end{equation}
since $d\geq 3$. The proof is complete.
\end{proof}

Before presenting the proof of the main theorem, we decompose the error as
\begin{equation}
\begin{aligned}
u_\eps(t,x,\omega)-u_{\hom}(t,x)=&\frac{1}{(2\pi)^d}\int_{\R^d} \hat{f}(\xi)e^{i\xi \cdot x}\E_B\{e^{i R_t^\eps(\lambda)} e^{iM_t^\eps(\lambda)}\} \, d\xi\\
&-\frac{1}{(2\pi)^d}\int_{\R^d}\hat{f}(\xi)e^{i\xi \cdot x}e^{-\frac12\xi^T\bar{A}\xi t} \, d\xi.
\end{aligned}
\end{equation}
Since $u_\eps-u_{\hom}$ does not depend on $\lambda$, we can send $\lambda\to 0$ on the r.h.s.\ of the above display. By Theorem~\ref{t:bdCo}, $R_t^\eps(\lambda)\to R_t^\eps$ and $M_t^\eps(\lambda)\to M_t^\eps$ in $L^2(\Omega\times \Sigma)$, where
\begin{eqnarray}
R_t^\eps:&=&-\eps\phi_{\xi}(\omega_{t/\eps^2})+\eps\phi_{\xi}(\omega_0),\\
M_t^\eps:&=&\sum_{j=1}^d \eps\int_0^{t/\eps^2}\sum_{i=1}^d(D_i\phi_{\xi}(\omega_s)+\xi_i)\sigma_{ij}(\omega_s) \, dB_s^j.
\end{eqnarray}
Therefore, the error can be rewritten as
\begin{equation}
\label{e:decomp-error}
\begin{aligned}
u_\eps(t,x,\omega)-u_{\hom}(t,x)=&\frac{1}{(2\pi)^d}\int_{\R^d} \hat{f}(\xi)e^{i\xi \cdot x}\E_B\{(e^{i R_t^\eps}-1)e^{iM_t^\eps}\} \, d\xi\\
&+\frac{1}{(2\pi)^d}\int_{\R^d} \hat{f}(\xi)e^{i\xi \cdot x}(\E_B\{e^{iM_t^\eps}\}-e^{-\frac12\xi^T\bar{A}\xi t}) \, d\xi.
\end{aligned}
\end{equation}
The first part measures how small the remainder $R_t^\eps$ is, and the second part measures how close the martingale $M_t^\eps$ is to a Brownian motion. It turns out that the error coming from the remainder generates the random, centered fluctuation, while the error coming from the martingale is of lower order. We will analyze them separately in the following two sections.

\section{An analysis of the remainder}
\label{s:r}

We define the error coming from the remainder in \eqref{e:decomp-error} as
\begin{equation}
\cE_1:=\frac{1}{(2\pi)^d}\int_{\R^d} \hat{f}(\xi)e^{i\xi \cdot x}\E_B\{(e^{i R_t^\eps}-1)e^{iM_t^\eps}\} \, d\xi.
\end{equation}
Let $\phi=(\phi_{e_1},\ldots,\phi_{e_d})$. The goal of this section is to show
\begin{prop}
\label{p:conR}
\begin{equation*}
\E\{|\cE_1-\eps \nabla u_{\hom}(t,x)\cdot \phi(\tau_{-x/\eps}\omega)|\}\leq C(1+t^\frac12) \Ll|
\begin{array}{ll}
\eps^{\frac43} & \text{if } d =  3, \\
\eps^{\frac32}|\log\eps|^\frac12 & \text{if } d =  4, \\
\eps^\frac32 & \text{if } d \ge 5,
\end{array}
\Rr.
\end{equation*}
 where $C$ is some constant.
\end{prop}

Recall that $$R_t^\eps=-\eps\phi_{\xi}(\omega_{t/\eps^2})+\eps\phi_{\xi}(\omega_0).$$ 
By Theorem~\ref{t:bdCo} and the stationarity of $\omega_s$, we obtain that
\begin{equation}
\E\E_B\{|R_t^\eps|^4\}\les |\xi|^4\eps^4.
\end{equation}
Using the fact that $|e^{ix}-1-ix|\leq x^2$ and $\hat{f}(\xi)|\xi|^2\in L^1(\R^d)$, we derive
\begin{equation}
\label{esR1}
\E\{|\cE_1-\cE_2|^2\}\les \eps^4,
\end{equation}
where
\begin{equation}
\cE_2:=\frac{1}{(2\pi)^d}\int_{\R^d} \hat{f}(\xi)e^{i\xi \cdot x}\E_B\{iR_t^\eps e^{iM_t^\eps}\} \, d\xi.
\end{equation}

Now we only need to analyze $\cE_2$. The two terms in $R_t^\eps$ are analyzed separately. For $-\eps\phi_\xi(\omega_{t/\eps^2})$, we can use the variance decay of $\E_B\{\phi_\xi(\omega_t)\}$ when $t$ is large. For $\eps \phi_\xi(\omega_0)$, since it is independent of the Brownian path, we expect that $e^{iM_t^\eps}$ averages itself. This will be proved by applying a special case of a quantitative martingale central limit theorem, which we present as the following proposition.

\begin{prop} \cite[Theorem~3.2]{mourrat2012kantorovich}
\label{p:2ndQM}
If $M_t$ is a continuous martingale and $\langle M\rangle_t$ is its predictable quadratic variation, $W_t$ is a standard Brownian motion, then
\begin{equation}
\label{e:2ndQM1}
d_{1,k}(M_t,\sigma W_t)\leq (k\vee 1)\E\{ |\langle M\rangle_t-\sigma^2t|\},
\end{equation}
with the distance $d_{k}$ defined as
\begin{equation}
\label{e:2ndQM2}
d_{1,k}(X,Y)=\sup \{|\E\{f(X)-f(Y)\}|: f\in C_b^2(\R), \|f'\|\leq 1, \|f''\|_\infty\leq k\}.
\end{equation}
\begin{rem}
	\label{r:improved}
	In fact, the argument in \cite{mourrat2012kantorovich} simplifies when we assume (as we do here) that the martingale $M_t$ is continuous. In this case, the multiplicative constant $(k \vee 1)$ in \eqref{e:2ndQM1} can be replaced by $k$, and the condition $\|f'\|\le 1$ in \eqref{e:2ndQM2} can be dropped.
\end{rem}
\end{prop}

We also need the following second moment estimate of additive functionals of $\omega_s$.

\begin{lem}
\label{l:vbdd}
For any $f\in L^2(\Omega)$, we have
\begin{equation*}
\E\E_B\{(\int_0^t f(\omega_s) \, ds)^2\}
\leq 2t\int_0^t\E\{|\E_B\{f(\omega_{s/2})\}|^2\} \, ds.
\end{equation*}
\end{lem}

\begin{proof}
The proof is a standard calculation. First, by stationarity we have 
\begin{equation}
\begin{aligned}
\E\E_B\{(\int_0^t f(\omega_s) \, ds)^2\}=&2\int_{0\leq s\leq u\leq t} \E\E_B\{f(\omega_s)f(\omega_u)\} \, dsdu\\
=&2\int_{0\leq s\leq u\leq t}\E\E_B\{f(\omega_0)f(\omega_{u-s})\} \, dsdu.
\end{aligned}
\end{equation}
Secondly, we change variable $s\mapsto u-s$ and integrate in $u$ to obtain
\begin{equation}
2\int_{0\leq s\leq u\leq t}\E\E_B\{f(\omega_0)f(\omega_{u-s})\} \, dsdu
=2\int_0^t(t-s)\E\E_B\{f(\omega_0)f(\omega_s)\} \, ds.
\end{equation}
By reversibility we further derive
\begin{equation}
\begin{aligned}
2\int_0^t(t-s)\E\E_B\{f(\omega_0)f(\omega_s)\} \, ds
=&2\int_0^t(t-s)\E\{|\E_B\{f(\omega_{s/2})\}|^2\} \, ds\\
\leq &2t\int_0^t\E\{|\E_B\{f(\omega_{s/2})\}|^2\} \, ds.
\end{aligned}
\end{equation}
The proof is complete.
\end{proof}

Now we can combine \eqref{esR1} with the following Lemmas~\ref{l:esR11} and \ref{l:esR12} to complete the proof of Proposition~\ref{p:conR}.

\begin{lem}
\label{l:esR11}
\begin{equation*}
\E\{(\int_{\R^d}|\hat{f}(\xi)| |\E_B\{\phi_\xi(\omega_{t/\eps^2})e^{iM_t^\eps}\}| \, d\xi)^2\}\les 
\Ll|
\begin{array}{ll}
\eps^{\frac23} & \text{if } d =  3, \\
\eps|\log\eps| & \text{if } d =  4, \\
\eps & \text{if } d \ge 5.
\end{array}
\Rr.
\end{equation*}
\end{lem}

\begin{proof}
First, we have for any $u\in (0,t)$ that 
\begin{equation}
\E_B\{\phi_\xi(\omega_{t/\eps^2})e^{iM_u^\eps}\}=\E_B\{\E_B\{\phi_\xi(\omega_{t/\eps^2})|\mathcal{F}_{u/\eps^2}\}e^{iM_u^\eps}\},
\end{equation}
where $\mathcal{F}_s$ is the natural filtration associated with $B_s$. By the stationarity of $\omega_s$, we obtain
\begin{equation}
\begin{aligned}
\E\{|\E_B\{\phi_\xi(\omega_{t/\eps^2})e^{iM_u^\eps}\}|^2\}\leq& \E\E_B\{|\E_B\{\phi_\xi(\omega_{t/\eps^2})|\mathcal{F}_{u/\eps^2}\}|^2\}\\
=&
\E\{|\E_B\{\phi_\xi(\omega_{(t-u)/\eps^2})\}|^2\}.
\end{aligned}
\end{equation}


Secondly, we have
\begin{equation}
\begin{aligned}
&\int_{\R^d}|\hat{f}(\xi)|\E\E_B\{|\phi_\xi(\omega_{t/\eps^2})(e^{iM_t^\eps}-e^{iM_u^\eps})|^2\} \, d\xi 
\\
\leq  &\int_{\R^d}|\hat{f}(\xi)|\E\E_B\{|\phi_\xi(\omega_{t/\eps^2})|^2|M_t^\eps-M_u^\eps|^2\} \, d\xi
\\
\leq & \int_{\R^d}|\hat{f}(\xi)|\sqrt{\E\E_B\{|\phi_\xi(\omega_{t/\eps^2})|^4\}}\sqrt{\E\E_B\{|M_t^\eps-M_u^\eps|^4\}} \, d\xi.
\end{aligned}
\end{equation}
By moment bounds of $\phi_\xi$, the first factor $\sqrt{\E\E_B\{|\phi_\xi(\omega_{t/\eps^2})|^4\}}\les |\xi|^2$. For the second factor, we apply moment inequalities of martingales to derive
\begin{equation}
\sqrt{\E\E_B\{|M_t^\eps-M_u^\eps|^4\}}\les \sqrt{\E\E_B\{|\langle M^\eps\rangle_t-\langle M^\eps\rangle_u|^2\}},
\end{equation}
with $\langle M^\eps\rangle_t$ the quadratic variation of $M_t^\eps$:
\begin{equation}
\begin{aligned}
\langle M^\eps\rangle_t=&\sum_{j=1}^d \eps^2\int_0^{t/\eps^2}\left(\sum_{i=1}^d (D_i\phi_\xi(\omega_s)+\xi_i)\sigma_{ij}(\omega_s)\right)^2 \,ds\\
=&\eps^2\int_0^{t/\eps^2}(\xi+D\phi_{\xi}(\omega_s))^Ta(\omega_s)(\xi+D\phi_\xi(\omega_s)) \,ds.
\end{aligned}
\end{equation}
By moment bounds of $D\phi_{\xi}$, we have $\sqrt{\E\E_B\{|M_t^\eps-M_u^\eps|^4\}}\les (t-u)|\xi|^2$. Therefore, we have obtained
\begin{equation}
\int_{\R^d}|\hat{f}(\xi)|\E\E_B\{|\phi_\xi(\omega_{t/\eps^2})(e^{iM_t^\eps}-e^{iM_u^\eps})|^2\} \, d\xi \les t-u.
\end{equation}

Now we can write 
\begin{equation}
\E_B\{\phi_\xi(\omega_{t/\eps^2})e^{iM_t^\eps}\}=\E_B\{\phi_\xi(\omega_{t/\eps^2})(e^{iM_t^\eps}-e^{iM_u^\eps})\}+\E_B\{\phi_\xi(\omega_{t/\eps^2})e^{iM_u^\eps}\},
\end{equation} and derive
\begin{equation}
\int_{\R^d}|\hat{f}(\xi)|\E\{|\E_B\{\phi_\xi(\omega_{t/\eps^2})e^{iM_t^\eps}\}|^2\} \, d\xi\les t-u+\int_{\R^d}|\hat{f}(\xi)|\E\{|\E_B\{\phi_\xi(\omega_{(t-u)/\eps^2})\}|^2\} \, d\xi.
\end{equation}
By Proposition~\ref{p:vdphi},
\begin{equation}
\int_{\R^d}|\hat{f}(\xi)|\E\{|\E_B\{\phi_\xi(\omega_{t/\eps^2})e^{iM_t^\eps}\}|^2\} \, d\xi\les t-u+\Ll|
\begin{array}{ll}
(\frac{\eps^2}{t-u})^{\frac12} & \text{if } d =  3, \\
\frac{\eps^2}{t-u}\log(2+\frac{t-u}{\eps^2}) & \text{if } d =  4, \\
\frac{\eps^2}{t-u} & \text{if } d \ge 5.
\end{array}
\Rr.
\end{equation}
After optimizing with respect to $u$ on the r.h.s.\ of the above display, we complete the proof.
\end{proof}

\begin{lem}
\label{l:esR12}
\begin{equation*}
\E\{|\frac{1}{(2\pi)^d}\int_{\R^d}\hat{f}(\xi)e^{i\xi\cdot x}i\eps \phi_\xi(\omega_0)\E_B\{e^{iM_t^\eps}\} \, d\xi-\eps \nabla u_{\hom}(t,x)\cdot \phi(\tau_{-x/\eps}\omega)|\}\les \eps^2\sqrt{t}.
\end{equation*}
\end{lem}

\begin{proof}
For almost every fixed $\omega\in \Omega$ and $\eps>0$, 
\begin{equation}
M_t^\eps=\sum_{j=1}^d \eps\int_0^{t/\eps^2}\sum_{i=1}^d(D_i\phi_{\xi}(\omega_s)+\xi_i)\sigma_{ij}(\omega_s) \, dB_s^j
\end{equation}
is a continuous square integrable martingale on $(\Sigma,\mathcal{A},\P_B)$, so by Proposition~\ref{p:2ndQM}, we have
\begin{equation}
|\E_B\{e^{iM_t^\eps}\}-e^{-\frac12\sigma_\xi^2t}|\leq \E_B\{|\langle M^\eps\rangle_t-\sigma_\xi^2 t|\},
\end{equation}
where $\langle M^\eps\rangle_t$ is the quadratic variation of $M_t^\eps$:
\begin{equation}
\begin{aligned}
\langle M^\eps\rangle_t=&\sum_{j=1}^d \eps^2\int_0^{t/\eps^2}\left(\sum_{i=1}^d (D_i\phi_\xi(\omega_s)+\xi_i)\sigma_{ij}(\omega_s)\right)^2 \, ds\\
=&\eps^2\int_0^{t/\eps^2}(\xi+D\phi_{\xi}(\omega_s))^Ta(\omega_s)(\xi+D\phi_\xi(\omega_s)) \, ds,
\end{aligned}
\end{equation}
and $\sigma_\xi^2=\xi^T\bar{A}\xi$, with the homogenized matrix $\bar{A}$ given by \eqref{hoMa}. 

Thus we have derived
\begin{equation*}
\begin{aligned}
&\int_{\R^d}|\hat{f}(\xi)|\E\{|\phi_\xi(\omega_0)(\E_B\{e^{iM_t^\eps}\}-e^{-\frac12\sigma_\xi^2t})|\} \, d\xi\\
 \les &\int_{\R^d}|\hat{f}(\xi)||\xi|\sqrt{\E\E_B\{|\langle M^\eps\rangle_t-\sigma_\xi^2 t|^2\}} \, d\xi.
\end{aligned}
\end{equation*}
By recalling \eqref{defpsi}, $\langle M^\eps\rangle_t-\sigma_\xi^2 t=\eps^2\int_0^{t/\eps^2}\psi_\xi(\omega_s) \, ds$, so we apply Lemma~\ref{l:vbdd} and Proposition~\ref{p:vdpsi} to obtain
\begin{align}
\label{bdpsi}
\E\E_B\{|\langle M^\eps\rangle_t-\sigma_\xi^2 t|^2\} 
& = \E\E_B\{|\eps^2\int_0^{t/\eps^2}\psi_\xi(\omega_s) \, ds|^2\} \notag \\
& \les \eps^2t\int_0^{t/\eps^2}\E\{|\E_B\{\psi_\xi(\omega_{s/2})\}|^2\} \, ds \notag \\
& \les \eps^2t|\xi|^4.
\end{align}

To summarize, we have
\begin{equation}
\begin{aligned}
&\E\{|\frac{1}{(2\pi)^d}\int_{\R^d}\hat{f}(\xi)e^{i\xi\cdot x}i\eps \phi_\xi(\omega_0)(\E_B\{e^{iM_t^\eps}\}-e^{-\frac12\sigma_\xi^2t}) \, d\xi|\}\\
\les &\eps\int_{\R^d}|\hat{f}(\xi)||\xi|\sqrt{\E\E_B\{|\langle M^\eps\rangle_t-\sigma_\xi^2 t|^2\}} \, d\xi
\les  \eps^2\sqrt{t}.
\end{aligned}
\end{equation}
Since $\phi_\xi=\sum_{k=1}^d \xi_k\phi_{e_k}$ and $\omega_0=\tau_{-X_0^\omega}\omega=\tau_{-x/\eps}\omega$, it is straightforward to check that
\begin{equation}
\frac{1}{(2\pi)^d}\int_{\R^d}\hat{f}(\xi)e^{i\xi\cdot x}i\eps \phi_\xi(\omega_0)e^{-\frac12\sigma_\xi^2t} \, d\xi=\eps\nabla u_{\hom}(t,x)\cdot \phi(\tau_{-x/\eps}\omega).
\end{equation}
The proof is complete.
\end{proof}


\section{An analysis of the martingale}
\label{s:m}

We define the error coming from the martingale part in \eqref{e:decomp-error} as
\begin{equation}
\begin{aligned}
\cE_3:=&\frac{1}{(2\pi)^d}\int_{\R^d} \hat{f}(\xi)e^{i\xi \cdot x}(\E_B\{e^{iM_t^\eps}\}-e^{-\frac12\xi^T\bar{A}\xi t}) \, d\xi\\
=&\frac{1}{(2\pi)^d}\int_{\R^d} \hat{f}(\xi)e^{i\xi \cdot x}(\E_B\{e^{iM_t^\eps}\}-e^{-\frac12\sigma_\xi^2 t}) \, d\xi.
\end{aligned}
\end{equation}
By the estimate in \eqref{bdpsi}, we already have
\begin{equation}
\label{bd2ndM}
\begin{aligned}
\E\{|\cE_3|^2\}\les &\int_{\R^d} |\hat{f}(\xi)| \E\{|\E_B\{e^{iM_t^\eps}\}-e^{-\frac12\sigma_\xi^2 t}|^2\} \, d\xi\\
\leq &\int_{\R^d}|\hat{f}(\xi)| \E\E_B\{|\langle M^\eps\rangle_t-\sigma_\xi^2t|^2\} \, d\xi \les \eps^2t.
\end{aligned}
\end{equation}
Thus $\cE_3$ is of order at most $\eps$, and we need to refine this estimate to show that it is actually of lower order.
The following is the main result of this section.
\begin{prop}
\label{p:conM}
\begin{equation}
\label{e:conM}
\E\{|\cE_3|\}\leq \eps C_\eps(t)
\end{equation}
 with $C_\eps(t)\to 0$ as $\eps\to0$ and $C_\eps(t)\leq C(1+t)$ for some constant $C>0$.
\end{prop}

The proof of Proposition~\ref{p:conM} can be decomposed into two parts. One part consists in showing that \eqref{e:conM} holds with $\cE_3$ replaced by 
$$
\cE_3-\eps t\sum_{i,j,k=1}^d c_{ijk}\partial_{x_ix_jx_k}u_{\hom}(t,x)
$$
for some constants $c_{ijk}$ defined below, see \eqref{defB}. In other words, 
$$
\eps t\sum_{i,j,k=1}^d c_{ijk}\partial_{x_ix_jx_k}u_{\hom}(t,x)
$$
is what we find to be the deterministic error at the order of $\eps$. The second part consists in observing that actually, the constants $c_{ijk}$ are all equal to zero!

We begin by defining $c_{ijk}$, and then observing that they are in fact zero. The following lemma from the proof of \cite[Theorem~1.8]{kipnis1986central} is needed, and we present a proof here for the sake of convenience.
\begin{lem}
\label{l:kv}
For any $V\in L^2(\Omega)$ with mean zero, let $\varphi_\lambda$ be the regularized corrector, i.e., $(\lambda-L)\varphi_\lambda=V$.
If 
\begin{equation*}
\E\E_B\Ll\{\left(\frac{1}{t^\frac12}\int_0^tV(\omega_s) \, ds\right)^2\Rr\}\leq C
\end{equation*}
for some constant $C>0$ independent of $t$, then $\lambda\langle \varphi_\lambda,\varphi_\lambda\rangle\to 0$ and $D_k\varphi_\lambda$ converges in $L^2(\Omega)$, $k=1,\ldots,d$.
\end{lem}

\begin{proof}
First, by the calculation in Lemma~\ref{l:vbdd}, we have
\begin{equation}
\E\E_B\{\left(\frac{1}{t^\frac12}\int_0^tV(\omega_s) \, ds\right)^2\}=\frac{2}{t}\int_0^t\int_0^s\langle e^{uL}V,V\rangle \, duds.
\end{equation}
Since $\int_0^s\langle e^{uL}V,V\rangle \, du$ is non-decreasing as a function of $s$, the l.h.s.\ of the above display being bounded is equivalent with $\int_0^\infty\langle e^{uL}V,V\rangle \, du<\infty$, i.e. $\langle V,(-L)^{-1}V\rangle<\infty$. Let $U(d\xi)$ be the projection valued measure associated with $-L$, i.e., $-L=\int_0^\infty \xi \, U(d\xi)$, and $\nu(d\xi)$ be the spectral measure associated with $V$, i.e. $\nu(d\xi)=\langle U(d\xi) V,V\rangle$. The fact that $\langle V,(-L)^{-1}V\rangle<\infty$ is equivalent to
\begin{equation}
\int_0^\infty \frac{1}{\xi} \, \nu(d\xi)<\infty.
\end{equation}
It follows that
\begin{equation}
\lambda\langle\varphi_\lambda,\varphi_\lambda\rangle=\int_0^\infty \frac{\lambda}{(\lambda+\xi)^2} \, \nu(d\xi)\to 0
\end{equation}
as $\lambda\to 0$ by the dominated convergence theorem. By the uniform ellipticity, we have
\begin{equation}
\langle D(\varphi_{\lambda_1}-\varphi_{\lambda_2}),D(\varphi_{\lambda_1}-\varphi_{\lambda_2})\rangle\les \langle \varphi_{\lambda_1}-\varphi_{\lambda_2},-L(\varphi_{\lambda_1}-\varphi_{\lambda_2})\rangle,
\end{equation}
and since 
\begin{equation}
\langle \varphi_{\lambda_1},-L\varphi_{\lambda_2}\rangle=\int_0^\infty \frac{\xi}{(\lambda_1+\xi)(\lambda_2+\xi)} \, \nu(d\xi)\to \int_0^\infty \frac{1}{\xi} \, \nu(d\xi)
\end{equation}
as $\lambda_1,\lambda_2\to 0$, we further obtain 
\begin{equation}
\langle D(\varphi_{\lambda_1}-\varphi_{\lambda_2}),D(\varphi_{\lambda_1}-\varphi_{\lambda_2})\rangle\to 0.
\end{equation}

The proof is complete.
\end{proof}

For $\psi_{ij}=(e_i+D\phi_{e_i})^Ta(e_j+D\phi_{e_j})-\bar{A}_{ij} $ ($ i,j=1,\ldots,d$), a polarization of the inequality in \eqref{bdpsi} ensures that
\begin{equation}
\label{fiaVar}
\E\E_B\{|\eps\int_0^{t/\eps^2}\psi_{ij}(\omega_s) \, ds|^2\}\les t,
\end{equation}
i.e., the asymptotic variance is finite, so we can apply Lemma~\ref{l:kv}: letting $\Psi_{\lambda,ij}$ be the regularized corrector associated with $\psi_{ij}$, i.e.,
\begin{equation}
(\lambda-L)\Psi_{\lambda,ij}=\psi_{ij},
\end{equation}
we have $\lambda\langle \Psi_{\lambda,ij},\Psi_{\lambda,ij}\rangle\to0$ as $\lambda\to 0$. We also have the convergence of $D_k\Psi_{\lambda,ij}$ in $L^2(\Omega)$, with the limit formally written as $D_k\Psi_{ij}:=\lim_{\lambda\to 0}D_k\Psi_{\lambda,ij}$.

Let $D\Psi_{ij}=(D_1\Psi_{ij},\ldots,D_d\Psi_{ij})$, then the constant $c_{ijk}$ for $i,j,k=1,\ldots,d$ is given by
\begin{equation}
\label{defB}
c_{ijk}:=\frac12\E\{(D\Psi_{ij})^Ta(e_k+D\phi_{e_k})\}.
\end{equation}

\begin{lem}
$c_{ijk}=0$ for $i,j,k=1,\ldots,d$.
\end{lem}

\begin{proof}
By the $L^2$ convergence of $D\Psi_{\lambda,ij}\to D\Psi_{ij}$ and $D\phi_{\lambda,e_k}\to D\phi_{e_k}$, we have
\begin{equation}
c_{ijk}=\lim_{\lambda\to 0}\frac12\E\{(D\Psi_{\lambda,ij})^Ta(e_k+D\phi_{\lambda,e_k})\}.
\end{equation}
An integration by parts leads to
\begin{equation}
\frac12\E\{(D\Psi_{\lambda,ij})^Ta(e_k+D\phi_{\lambda,e_k})\}=\langle \Psi_{\lambda,ij}, \frac12\sum_{m,n=1}^d D_m(a_{mn}D_n\phi_{\lambda,e_k})+\frac12\sum_{m=1}^dD_ma_{mk}\rangle.
\end{equation}
The r.h.s.\ of the above display can be rewritten as $\langle \Psi_{\lambda,ij},L\phi_{\lambda,e_k}+e_k\cdot b\rangle$, and by recalling the equation satisfied by the regularized corrector \eqref{corEqP}, we have
\begin{equation}
\frac12\E\{(D\Psi_{\lambda,ij})^Ta(e_k+D\phi_{\lambda,e_k})\}=
\langle \Psi_{\lambda,ij},\lambda\phi_{\lambda,e_k}\rangle,
\end{equation}
which goes to zero as $\lambda\to 0$. The proof is complete.
\end{proof}

To refine the estimation of $\cE_3$, we need a more accurate estimation of $\E_B\{e^{iM_t^\eps}\}-e^{-\frac12\sigma_\xi^2 t}$ compared with the one obtained by Proposition~\ref{p:2ndQM}. This is given by the following quantitative martingale central limit theorem.

\begin{prop}\cite[Proposition~3.2]{gu2014fluctuations}
\label{prop:3rdQM}
If $M_t$ is a continuous martingale and $\langle M\rangle_t$ is its predictable quadratic variation, $W_t$ is a standard Brownian motion, then for any $f\in \mathcal{C}_b(\R)$ with up to third order bounded and continuous derivatives, we have
\begin{equation*}
|\E\{f(M_t)-f(\sigma W_t)-\frac12f''(M_\tau)(\langle M\rangle_t-\sigma^2 t)\}|\leq C\|f'''\|_\infty\E\{|\langle M\rangle_t-\sigma^2t |^{\frac32}\},
\end{equation*}
where $\tau=\sup\{s\in [0,t]|\langle M\rangle_s\leq \sigma^2t\}$, $\|f'''\|_\infty$ denotes the supreme bound of $f'''$, and $C$ is some universal constant.
\end{prop}

\begin{rem}
\label{}
In the discrete-space setting, the corresponding martingales have jumps, and we do not know how to adapt Proposition~\ref{prop:3rdQM} and the subsequent argument to recover Theorem~\ref{t:main} in this case.
\end{rem}

By the above proposition, we have for almost every $\omega\in\Omega$ that
\begin{equation}
|\E_B\{e^{iM_t^\eps}\}-e^{-\frac12\sigma_\xi^2 t}+\frac12\E_B\{e^{iM_\tau^\eps}(\langle M^\eps\rangle_t-\sigma_\xi^2t)\}|\leq C\E_B\{|\langle M^\eps\rangle_t-\sigma_\xi^2t|^\frac32\},
\end{equation}
where 
\begin{equation}
\tau=\sup\{s\in [0,t]: \sum_{j=1}^d \eps^2\int_0^{s/\eps^2}\left(\sum_{i=1}^d (D_i\phi_\xi(\omega_u)+\xi_i)\sigma_{ij}(\omega_u)\right)^2 \, du\le \sigma_\xi^2 t\}.
\end{equation}
Combining with \eqref{bdpsi}, we obtain
\begin{equation}
\label{esM11}
\E\{|\cE_3-\cE_4|\}\les \int_{\R^d}|\hat{f}(\xi)|\E\E_B\{|\langle M^\eps\rangle_t-\sigma_\xi^2t|^\frac32\} \, d\xi \les \eps^\frac32t^\frac34
\end{equation}
for
\begin{equation}
\begin{aligned}
\cE_4:=&-\frac{1}{2(2\pi)^d}\int_{\R^d}\hat{f}(\xi)e^{i\xi\cdot x}\E_B\{e^{iM_\tau^\eps}(\langle M^\eps\rangle_t-\sigma_\xi^2t)\} \, d\xi\\
=&-\frac{1}{2(2\pi)^d}\int_{\R^d}\hat{f}(\xi)e^{i\xi\cdot x}\E_B\{e^{iM_\tau^\eps}\eps^2\int_0^{t/\eps^2}\psi_\xi(\omega_s)ds\} \, d\xi.
\end{aligned}
\end{equation}

Define 
\begin{equation}
\cE_5:=-\frac{1}{2(2\pi)^d}\int_{\R^d}\hat{f}(\xi)e^{i\xi\cdot x}\E_B\{e^{iM_t^\eps}\eps^2\int_0^{t/\eps^2}\psi_\xi(\omega_s)ds\} \, d\xi.
\end{equation}
The following Lemmas~\ref{l:esM11} and \ref{l:esM12} combine with \eqref{esM11} to complete the proof of Proposition~\ref{p:conM}.

\begin{lem}
\label{l:esM11}
$\E\{|\cE_4-\cE_5|\}\les \eps^\frac32t^\frac34$.
\end{lem}

\begin{proof}
By \eqref{bdpsi}, we know $\E\E_B\{|\eps^2\int_0^{t/\eps^2}\psi_\xi(\omega_s) \, ds|^2\}\les \eps^2t|\xi|^4$, so
\begin{equation}
\E\{|\cE_4-\cE_5|\}\les \eps t^\frac12\int_{\R^d}|\hat{f}(\xi)||\xi|^2\sqrt{\E\E_B\{|M_\tau^\eps-M_t^\eps|^2\}} \, d\xi.
\end{equation}
By the definition of $\tau$, we have 
\begin{equation}
\int_{\R^d}|\hat{f}(\xi)||\xi|^2\sqrt{\E\E_B\{|M_\tau^\eps-M_t^\eps|^2\}} \, d\xi\les  \int_{\R^d}|\hat{f}(\xi)||\xi|^2\sqrt{\E\E_B\{|\sigma_\xi^2t-\langle M^\eps\rangle_t|\}} \, d\xi\les \eps^\frac12t^\frac14,
\end{equation}
so $\E\{|\cE_4-\cE_5|\}\les \eps^\frac32t^\frac34$. The proof is complete.
\end{proof}

\begin{lem}
\label{l:esM12}
\begin{equation*}
\E\{|\cE_5-\eps t\sum_{i,j,k=1}^d c_{ijk}\partial_{x_ix_jx_k}u_{\hom}(t,x)|\}\leq \eps C_\eps(t),
\end{equation*} 
with $C_\eps(t)\to 0$ as $\eps\to 0$ and $C_\eps(t)\leq C(1+t)$ for some constant $C$.
\end{lem}

\begin{proof}
We write
\begin{equation}
\frac{\cE_5}{\eps}=-\frac{1}{2(2\pi)^d}\int_{\R^d}\hat{f}(\xi)e^{i\xi\cdot x}\E_B\{e^{iM_t^\eps}\eps\int_0^{t/\eps^2}\psi_\xi(\omega_s) \, ds\} \, d\xi,
\end{equation}
where $\eps\int_0^{t/\eps^2}\psi_\xi(\omega_s) \, ds$ is of central limit scaling. To apply the Kipnis-Varadhan method, the only condition we need to check is the finiteness of the asymptotic variance, and this is already given by \eqref{bdpsi}, i.e. we have
\begin{equation}
\E\E_B\{|\eps\int_0^{t/\eps^2}\psi_\xi(\omega_s) \, ds|^2\}\les t|\xi|^4.
\end{equation}
Therefore, we can write $\eps\int_0^{t/\eps^2}\psi_\xi(\omega_s) \, ds=\mathcal{R}_t^\eps+\mathcal{M}_t^\eps$ with  
\begin{equation}
\begin{aligned}
\mathcal{R}_t^\eps=&\eps\int_0^{t/\eps^2}\lambda\Psi_{\lambda,\xi}(\omega_s) \, ds-\eps\Psi_{\lambda,\xi}(\omega_{t/\eps^2})+\eps\Psi_{\lambda,\xi}(\omega_0)\\
&+\sum_{j=1}^d \eps\int_0^{t/\eps^2}\sum_{i=1}^d (D_i\Psi_{\lambda,\xi}(\omega_s)-D_i\Psi_\xi(\omega_s))\sigma_{ij}(\omega_s) \, dB_s^j,
\end{aligned}
\end{equation}
and
\begin{equation}
\mathcal{M}_t^\eps=\sum_{j=1}^d \eps\int_0^{t/\eps^2}\sum_{i=1}^d D_i\Psi_\xi(\omega_s)\sigma_{ij}(\omega_s) \, dB_s^j.
\end{equation}
Recall that the formally-written random variable $D_i\Psi_\xi$ is the $L^2$-limit of $D_i\Psi_{\lambda,\xi}$ as $\lambda\to 0$, with $\Psi_{\lambda,\xi}$ solving the regularized corrector equation 
\begin{equation}
(\lambda-L)\Psi_{\lambda,\xi}=\psi_\xi.
\end{equation}
Since $\psi_\xi=\sum_{i,j=1}^d\xi_i\xi_j\psi_{ij}$, by linearity we have $\Psi_{\lambda,\xi}=\sum_{i,j=1}^d \xi_i\xi_j\Psi_{\lambda,ij}$, with $\Psi_{\lambda,ij}$ solving
\begin{equation}
(\lambda-L)\Psi_{\lambda,ij}=\psi_{ij}.
\end{equation}

Now we can write
\begin{equation}
\frac{\cE_5}{\eps}=-\frac{1}{2(2\pi)^d}\int_{\R^d}\hat{f}(\xi)e^{i\xi\cdot x}\E_B\{e^{iM_t^\eps}(\mathcal{R}_t^\eps+\mathcal{M}_t^\eps)\}d\xi.
\end{equation}

First, by choosing $\lambda=\eps^2$ and using the stationarity of $\omega_s$ we have

\begin{equation}
\E\E_B\{|\eps\int_0^{t/\eps^2}\lambda\Psi_{\lambda,\xi}(\omega_s) \, ds-\eps\Psi_{\lambda,\xi}(\omega_{t/\eps^2})+\eps\Psi_{\lambda,\xi}(\omega_0)|^2\}\les \lambda\langle\Psi_{\lambda,\xi},\Psi_{\lambda,\xi}\rangle(1+t^2).
\end{equation}
For the stochastic integral, we have
\begin{equation}
\begin{aligned}
&\E\E_B\{|\sum_{j=1}^d \eps\int_0^{t/\eps^2}\sum_{i=1}^d (D_i\Psi_{\lambda,\xi}(\omega_s)-D_i\Psi_\xi(\omega_s))\sigma_{ij}(\omega_s) \, dB_s^j|^2\}\\
=&\E\E_B\{\sum_{j=1}^d \eps^2\int_0^{t/\eps^2}(\sum_{i=1}^d (D_i\Psi_{\lambda,\xi}(\omega_s)-D_i\Psi_\xi(\omega_s))\sigma_{ij}(\omega_s))^2 \, ds\}\\
\les & \sum_{i=1}^d\langle D_i\Psi_{\lambda,\xi}-D_i\Psi_\xi,D_i\Psi_{\lambda,\xi}-D_i\Psi_\xi\rangle t.
\end{aligned}
\end{equation}
Therefore,
\begin{equation}
\begin{aligned}
\E\E_B\{|\mathcal{R}_t^\eps|^2\}\les& \lambda\langle \Psi_{\lambda,\xi},\Psi_{\lambda,\xi}\rangle(1+t^2)\\
&+\sum_{i=1}^d\langle D_i\Psi_{\lambda,\xi}-D_i\Psi_\xi,D_i\Psi_{\lambda,\xi}-D_i\Psi_\xi\rangle t.
\end{aligned}
\end{equation}
By Lemma~\ref{l:kv}, $\lambda\langle\Psi_{\lambda,\xi},\Psi_{\lambda,\xi}\rangle\to0$ as $\lambda\to 0$, and $D_i\Psi_{\lambda,\xi}\to D_i\Psi_\xi$ in $L^2(\Omega)$, so we derive
\begin{equation}
\label{conRM}
\int_{\R^d}|\hat{f}(\xi)| \E\E_B\{|\mathcal{R}_t^\eps|\} \, d\xi\leq C_\eps(1+t)
\end{equation}
with $C_\eps\to 0$ as $\eps\to 0$.

Secondly, for the martingale part $\E_B\{e^{iM_t^\eps}\mathcal{M}_t^\eps\}$, it is clear that $M_t^\eps$ and $\mathcal{M}_t^\eps$ are written as $\sum_{j=1}^d\eps\int_0^{t/\eps^2}f_j(\omega_s) \, dB_s^j$ and $\sum_{j=1}^d\eps\int_0^{t/\eps^2}g_j(\omega_s) \, dB_s^j$ for some $f_j,g_j\in L^2(\Omega)$ respectively. We claim that for fixed $\xi\in \R^d,t>0$
\begin{equation}
\label{con2M}
\E\{|\E_B\{e^{iM_t^\eps}\mathcal{M}_t^\eps\}-c_\xi|\}\to 0
\end{equation}
for some constant $c_\xi$.

Recall that $\omega_s$ depends on $\eps$ through the initial condition $\omega_0=\tau_{-x/\eps}\omega$. By stationarity we can shift the environment $\omega$ by an amount of $x/\eps$ without changing the value of $\E\{|\E_B\{e^{iM_t^\eps}\mathcal{M}_t^\eps\}-c_\xi|\}$. So we can assume $\omega_s=\tau_{X_s^\omega}\omega$ with $X_0^\omega=0$.

For almost every $\omega\in\Omega$, by ergodicity we have
\begin{eqnarray*}
\sum_{j=1}^d \eps^2\int_0^{t/\eps^2}f_j^2(\omega_s) \, ds\to t\sum_{j=1}^d \langle f_j,f_j\rangle\\
\sum_{j=1}^d \eps^2\int_0^{t/\eps^2}g_j^2(\omega_s) \, ds\to t\sum_{j=1}^d \langle g_j,g_j\rangle\\
\sum_{j=1}^d \eps^2\int_0^{t/\eps^2}f_j(\omega_s)g_j(\omega_s) \, ds\to t\sum_{j=1}^d \langle f_j,g_j\rangle
\end{eqnarray*}
almost surely in $\Sigma$. Thus by a martingale central limit theorem \cite[page 339, Theorem~1.4]{ethier2009markov}, we have that for almost every $\omega\in\Omega$,
\begin{equation}
(M_t^\eps,\mathcal{M}_t^\eps)\Rightarrow (N_1,N_2)
\end{equation}
in distribution in $\Sigma$, where $(N_1,N_2)$ is a Gaussian vector with mean zero and whose covariance matrix is determined by $\E\{N_1^2\}=t\sum_{j=1}^d \langle f_j,f_j\rangle$, $\E\{N_2^2\}=t\sum_{j=1}^d \langle g_j,g_j\rangle$, and $\E\{N_1N_2\}=t\sum_{j=1}^d \langle f_j,g_j\rangle$. 

Now let $g_K(x)=(x\wedge K)\vee (-K)$ be a continuous and bounded cutoff function for $K>0$, and $h_K(x)=x-g_K(x)$ we have 
\begin{equation}
\E_B\{e^{iM_t^\eps}\mathcal{M}_t^\eps\}=\E_B\{e^{iM_t^\eps}g_K(\mathcal{M}_t^\eps)\}+\E_B\{e^{iM_t^\eps}h_K(\mathcal{M}_t^\eps)\}
\end{equation}
It is clear that $\E\E_B\{|\mathcal{M}_t^\eps|^2\}\les t|\xi|^4$, so 
\begin{equation}
\E\E_B\{|h_K(\mathcal{M}_t^\eps)|\}\leq \E\E_B\{|\mathcal{M}_t^\eps|1_{|\mathcal{M}_t^\eps|\geq K}\}\leq \frac{1}{K}\E\E_B\{|\mathcal{M}_t^\eps|^2\}\les \frac{t|\xi|^4}{K}.
\end{equation}
Therefore, 
\begin{equation}
\begin{aligned}
\limsup_{\eps\to 0}\E\{|\E_B\{e^{iM_t^\eps}\mathcal{M}_t^\eps\}-\E\{e^{iN_1}N_2\}|\}\leq& \lim_{\eps\to 0}\E\{|\E_B\{e^{iM_t^\eps}g_K(\mathcal{M}_t^\eps)\}-\E\{e^{iN_1}g_K(N_2)\}|\}\\
&+|\E\{e^{iN_1}h_K(N_2)\}|+\E\E_B\{|e^{iM_t^\eps}h_K(\mathcal{M}_t^\eps)|\}|\\
\les&|\E\{e^{iN_1}h_K(N_2)\}|+\frac{t|\xi|^4}{K}.
\end{aligned}
\end{equation}
Letting $K\to \infty$, \eqref{con2M} is proved for $c_\xi=\E\{e^{iN_1}N_2\}$. 


For the constant $c_\xi$, we have
\begin{equation}
c_\xi=ie^{-\frac12\E\{N_1^2\}}\E\{N_1N_2\}
\end{equation}
(this can be easily seen by differentiating the formula for $\E\{e^{i N_1 + i\zeta N_2}\}$ with respect to $\zeta$).
Recall that $f_j = \sum_{i=1}^d (D_i \phi_\xi + \xi_i) \sigma_{ij}$ and $g_j = \sum_{i = 1}^d D_i \Psi_{\xi} \sigma_{ij}$. After some calculation, we obtain
\begin{equation}
\sum_{j=1}^d\langle f_j,g_j\rangle=\sum_{i,j,k=1}^d \xi_i\xi_j\xi_k\E\{(D\Psi_{ij})^Ta(e_k+D\phi_{e_k})\},
\end{equation}
so, recalling \eqref{defB},
\begin{equation}
c_\xi=2ie^{-\frac12\sigma_\xi^2t}t\sum_{i,j,k=1}^d \xi_i\xi_j\xi_k c_{ijk}.
\end{equation}

By the above expression of $c_\xi$ and the fact that $\E\E_B\{|\mathcal{M}_t^\eps|^2\}\les t|\xi|^4$, we have
\begin{equation}
\E\{|\E_B\{e^{iM_t^\eps}\mathcal{M}_t^\eps\}-c_\xi|\}\les t^\frac12|\xi|^2+t|\xi|^3,
\end{equation}
so applying the dominated convergence theorem, we conclude that for $t>0$
\begin{equation}
\label{conMM}
\int_{\R^d}|\hat{f}(\xi)|\E\{|\E_B\{e^{iM_t^\eps}\mathcal{M}_t^\eps\}-c_\xi|\} \, d\xi \to 0
\end{equation}
as $\eps \to 0$.

To summarize, by combining \eqref{conRM} and \eqref{conMM} we have proved
\begin{equation}
\E\{|\frac{\cE_5}{\eps}+\frac{1}{2(2\pi)^d}\int_{\R^d}\hat{f}(\xi)e^{i\xi\cdot x}c_\xi \, d\xi|\}\to 0
\end{equation} 
as $\eps\to 0$, and the following bound holds
\begin{equation}
\E\{|\frac{\cE_5}{\eps}+\frac{1}{2(2\pi)^d}\int_{\R^d}\hat{f}(\xi)e^{i\xi\cdot x}c_\xi \, d\xi|\}\leq C(1+t)
\end{equation}
for some constant $C>0$ independent of $(t,x)$.

Now we only need to note that
\begin{equation}
\begin{aligned}
&-\frac{1}{2(2\pi)^d}\int_{\R^d}\hat{f}(\xi)e^{i\xi\cdot x}c_\xi \, d\xi\\
=&{t}\sum_{i,j,k=1}^d c_{ijk}\frac{1}{(2\pi)^d}\int_{\R^d}\hat{f}(\xi)(-i) \xi_i\xi_j\xi_ke^{i\xi\cdot x}e^{-\frac12\sigma_\xi^2t} \, d\xi\\
=&t\sum_{i,j,k=1}^dc_{ijk}\partial_{x_ix_jx_k}u_{\hom}(t,x)
\end{aligned}
\end{equation}
to complete the proof.
\end{proof}
\begin{rem}
\label{r:exponent-improvement}
From the proof above, we see that in order to estimate the rate of convergence to $0$ of $\E\{|C_\eps|\}$ in Theorem~\ref{t:main}, the rates of convergence of $\lambda\langle\Psi_{\lambda,\xi},\Psi_{\lambda,\xi} \rangle$ to $0$ and of $D\Psi_{\lambda,\xi}$ to $D\Psi_\xi$ as $\lambda \to 0$ need to be quantified. This in turn could be obtained by reinforcing Proposition~\ref{p:vdpsi} to 
\begin{equation}
\label{e.decay}
\E\{|\E_B\{\psi_\xi(\omega_t)\}|^2\} \lesssim t^{-\gamma},
\end{equation}
for some $\gamma > 1$. More precisely, spectral computations similar to those of \cite{mourrat2011variance} show that \eqref{e.decay} implies
$$
\lambda\langle\Psi_{\lambda,\xi},\Psi_{\lambda,\xi} \rangle \lesssim \lambda^{(\gamma-1) \wedge 1},
$$
and the same estimate for $\E\{|D \Psi_{\lambda,\xi}-D \Psi_\xi|^2\}$.
It was shown in \cite[Theorem~2.1]{gloria2011optimal} that the spatial averages of $\psi_\xi$ behave as if $\psi_\xi$ was a local function of the coefficient field. If $\psi_\xi$ is replaced by a truly local function, then the methods of \cite{gloria2013quantification} show that \eqref{e.decay} holds with $\gamma = d/2$. For our actual function $\psi_\xi$, it is thus natural to expect \eqref{e.decay} to hold at least for every $\gamma < d/2$, but a proof of this stronger result would require more work, so we preferred to present a simpler argument here.
\end{rem}

\section{Results on elliptic equations}
\label{s:ell}

The solutions to elliptic equations can be written as
\begin{eqnarray}
U_\eps(x,\omega)=\int_0^\infty e^{-t}u_\eps(t,x,\omega) \, dt\\
U_{\hom}(x)=\int_0^\infty e^{-t}u_{\hom}(t,x) \, dt.
\end{eqnarray}
Recall the error decomposition for fixed $(t,x)$ in the parabolic case
\begin{equation}
u_\eps(t,x,\omega)-u_{\hom}(t,x)
=\eps \nabla u_{\hom}(t,x)\cdot \phi(\tau_{-x/\eps}\omega)+\eps C_\eps(t,x),
\end{equation}
where $C_\eps(t,x)\to 0$ in $L^1(\Omega)$. By Propositions~\ref{p:conR} and \ref{p:conM}, we actually have
\begin{equation}
\E\{|C_\eps(t,x)|\}\leq C(1+t)
\end{equation}
for some constant $C>0$, so by the dominated convergence theorem 
\begin{equation}
\int_0^\infty e^{-t}\E\{|C_\eps(t,x)|\} \, dt\to 0
\end{equation}
as $\eps\to0$. Therefore, we obtain the error decomposition for fixed $x$ in the elliptic case
\begin{equation}
\label{errEll}
U_\eps(x,\omega)-U_{\hom}(x)
=
\eps \int_0^\infty e^{-t}\nabla u_{\hom}(t,x) \, dt\cdot \phi(\tau_{-x/\eps}\omega)+\eps \tilde{C}_\eps(x)
\end{equation}
with $\tilde{C}_\eps(x)\to 0$ in $L^1(\Omega)$.

The first term on the r.h.s.\ of \eqref{errEll} gives
\begin{equation}
\eps \int_0^\infty e^{-t}\nabla u_{\hom}(t,x) \, dt\cdot \phi(\tau_{-x/\eps}\omega)=\eps\nabla U_{\hom}(x)\cdot \phi(\tau_{-x/\eps}\omega),
\end{equation}
which completes the proof of Theorem~\ref{t:mainE}.

\section{Results for periodic coefficients}
\label{s:per}

It is natural to ask whether the same result holds for periodic rather than random coefficients. To understand the first order errors in periodic homogenization is a classical problem, however the \emph{pointwise} expansion proved in this paper does not seem to be known. Our approach applies with some minor modifications, which we now briefly discuss.

The existence of a ``stationary'' corrector now becomes trivial. We assume the coefficient $\tilde{a}(x)$ is defined on the $d-$dimensional torus $\mathbb{T}$, and by the fact that $\tilde{b}=(\tilde{b}_1,\ldots,\tilde{b}_d)$ with $\tilde{b}_i=\frac12\sum_{j=1}^d \partial_{x_j}\tilde{a}_{ji}$, we have
\[
\int_{\mathbb{T}} \tilde{b}(x)dx=0.
\]
By the Fredholm alternative, the corrector equation
\[
-\frac12\nabla\cdot \tilde{a}(x)\nabla \tilde{\phi}_\xi=\xi\cdot \tilde{b}
\]
has a unique solution satisfying $\int_{\mathbb{T}} \tilde{\phi}_\xi(x)dx=0$. The same discussion applies to $\tilde{\psi}_\xi=(\xi+\nabla\tilde{\phi}_\xi)^T\tilde{a}(\xi+\nabla\tilde{\phi}_\xi)-\xi^T\bar{A}\xi$ since $\int_{\mathbb{T}} \tilde{\psi}_\xi(x)dx=0$, that is, there exists a unique $\tilde{\Psi}_\xi$ solving
\[
-\frac12\nabla\cdot \tilde{a}(x)\nabla \tilde{\Psi}_\xi=\tilde{\psi}_\xi
\]
such that $\int_{\mathbb{T}} \tilde{\Psi}_\xi(x)dx=0$. Since we assume $\tilde{a}$ to be H\"older regular, the functions $\tilde{\phi}_\xi, \nabla \tilde{\phi}_\xi$ and $\tilde{\Psi}_\xi$ are bounded in $x$ (see \cite[Theorem~3.13]{han1997elliptic}).

Our estimates of variance decay in Propositions~\ref{p:vdphi} and \ref{p:vdpsi} can be replaced by a spectral gap inequality in the periodic setting. For the diffusion on the torus given by 
\[
dX_t=\tilde{b}(X_t)dt+\tilde{\sigma}(X_t)dB_t, 
\]
the Lebegue measure on $\mathbb{T}$ is the unique invariant measure and the following estimate holds \cite[Page 373, Theorem 3.2]{papanicolau1978asymptotic}: 
\begin{equation}
\sup_{X_0\in \mathbb{T}}|\E_B\{ g(X_t)\}|\les e^{-\rho t} \sup_{x\in \mathbb{T}}|g(x)|,
\label{eq:sgper}
\end{equation}
 for some $\rho>0$, provided $\int_{\mathbb{T}} g(x)dx=0$. This enables to replace the estimates of Propositions~\ref{p:vdphi} and \ref{p:vdpsi} by exponential bounds.
%
%
 
With the above two points in mind, we apply the same arguments to derive a result similar to Theorem~\ref{t:main}: for every fixed $(t,x)$, 
\[
 u_\eps(t,x)-u_{\hom}(t,x)
=  \eps \nabla u_{\hom}(t,x)\cdot \tilde{\phi}(\frac{x}{\eps})+o(\eps)
\]
where $\tilde{\phi}=(\tilde{\phi}_{e_1},\ldots,\tilde{\phi}_{e_d})$.

\section*{Acknowledgment}  
YG would like to thank his PhD advisor Prof. Guillaume Bal for many helpful discussions on the subject.

\appendix

\section{Estimating convolution of powers}

\begin{lem}
\label{l:conL}
When $d\geq 3$, for any $x\in \R^d$, 
\begin{eqnarray}
\label{conE1}
\sum_{k\in \Z^d}(1\wedge \frac{1}{|k|^{d-1}})(1\wedge \frac{1}{|x-k|^{d-1}})\les 1\wedge \frac{1}{|x|^{d-2}},\\
\label{conE2}
\sum_{k\in \Z^d}(1\wedge \frac{1}{|k|^{d}})(1\wedge \frac{1}{|x-k|^{d}})\les 1\wedge \frac{\log(2+|x|)}{|x|^{d}}.
\end{eqnarray}
\end{lem}

\begin{proof}
The proofs of \eqref{conE1} and \eqref{conE2} are similar, so we only consider \eqref{conE1}.

First, for $|x|>100$, we divide $\Z^d$ into three regions, $(I)=\{k\in \Z^d: |k|\leq |x|,|k|\leq |x-k|\}$, $(II)=\{k\in \Z^d: |k-x|\leq |x|,|k|>|x-k|\}$, $(III)=\{k\in \Z^d:|k|>|x|,|k-x|>|x|\}$. Then it is clear that in $(I)$, we have $|x-k|\geq |x|/2$, so
\begin{equation}
\sum_{k\in (I)}(1\wedge \frac{1}{|k|^{d-1}})(1\wedge \frac{1}{|x-k|^{d-1}})\les \sum_{|k|\leq |x|}(1\wedge \frac{1}{|k|^{d-1}}) \frac{1}{|x|^{d-1}}\les \frac{1}{|x|^{d-2}}.
\end{equation}
Similarly, in $(II)$ we have $|k|\geq |x|/2$, so
\begin{equation}
\sum_{k\in (II)}(1\wedge \frac{1}{|k|^{d-1}})(1\wedge \frac{1}{|x-k|^{d-1}})\les \sum_{|k-x|\leq |x|}(1\wedge \frac{1}{|x-k|^{d-1}})\frac{1}{|x|^{d-1}}\les \frac{1}{|x|^{d-2}}.
\end{equation}
In $(III)$, $|x-k|\geq |k|/2$, so
\begin{equation}
\sum_{k\in (III)}(1\wedge \frac{1}{|k|^{d-1}})(1\wedge \frac{1}{|x-k|^{d-1}})\les \sum_{|k|\geq |x|}\frac{1}{|k|^{2d-2}}\les \frac{1}{|x|^{d-2}}.
\end{equation}

Now for $|x|\leq 100$, it is clear that the summation is bounded since $d\geq 3$, so the proof of \eqref{conE1} is complete.
\end{proof}

\def\cprime{$'$}

\end{document}